\documentclass[a4paper]{article}

\usepackage[a4paper, total={6in, 8in}]{geometry}

\usepackage[utf8]{inputenc}

\usepackage{algorithm}
\usepackage{algorithmicx}
\usepackage{algpseudocode}
\usepackage{caption}
\usepackage{graphicx}
\usepackage{amsmath}
\usepackage{amssymb}
\usepackage{mathtools}
\usepackage{colonequals}
\usepackage{adjustbox}
\usepackage{subcaption}
\usepackage{xcolor}
\usepackage{tikz} 
\usepackage{amsthm}

\usepackage[toc,page]{appendix}
\usepackage{url}
\usepackage{booktabs}

\newtheorem{theorem}{Theorem}
\newtheorem{definition}{Definition}

\newtheorem{lemma}{Lemma}
\newtheorem{corollary}{Corollary}

\newtheorem{remark}{Remark}
\newtheorem{example}{Example}
\DeclareMathOperator{\sign}{sign}

\DeclareMathOperator*{\TV}{TV}
\DeclareMathOperator*{\TR}{TR}

\newcommand*{\logeq}{\ratio\Leftrightarrow}
\title{On Discrete Subproblems in Integer Optimal Control with Total Variation Regularization in Two Dimensions}
\date{March 2024}
\author{
	Paul Manns\\
	TU Dortmund University\\
	\texttt{paul.manns@tu-dortmund.de}
	\and
	Marvin Severitt\\
	TU Dortmund University\\
	\texttt{marvin.severitt@tu-dortmund.de}
}
\begin{document}
\maketitle

\begin{abstract}
We analyze integer linear programs which we obtain after discretizing two-dimensional subproblems arising from a trust-region algorithm for mixed integer optimal control problems with total variation regularization. We discuss NP-hardness of the discretized problems and the connection to graph-based problems. We show that the underlying polyhedron exhibits structural restrictions in its vertices with regards to which variables can attain fractional values at the same time. Based on this property, we derive cutting planes by employing a relation to shortest-path and minimum bisection problems. We propose a branching rule and a primal heuristic which improves previously found feasible points. We validate the proposed tools with a numerical benchmark in a standard integer programming solver. We observe a significant speedup for medium-sized problems. Our results give hints for scaling towards larger instances in the future.
\end{abstract}

\section{Introduction}
We concern ourselves with problems of the form 
\begin{gather}\label{eq:trip}
\begin{aligned}
    \min_{d}\quad & \sum_{i=1}^N \sum_{j=1}^M c_{i,j} d_{i,j} + \alpha  \sum_{j=1}^{M} \sum_{i=1}^{N-1}  \vert x_{i+1,j} + d_{i+1,j} - x_{i,j} - d_{i,j} \vert \\
    &+ \alpha  \sum_{i=1}^{N} \sum_{j=1}^{M-1}  \vert x_{i,j+1} + d_{i,j+1} - x_{i,j} - d_{i,j} \vert\\
    \text{s.t.}\quad
    & x_{i,j} + d_{i,j} \in \Xi \text{ for all } i \in [N], j \in [M],\quad \sum_{i=1}^N \sum_{j=1}^M  \vert  d_{i,j} \vert \le \Delta,
\end{aligned}
\tag{TR-IP}
\end{gather}
where $N,M \in \mathbb{N}$, $x \in \Xi^{N\times M}$, $c \in \mathbb{R}^{N\times M}$, $\alpha > 0$, and a finite set $\Xi \subset \mathbb{Z}$ are given. The problems \eqref{eq:trip} can be reformulated as integer (linear) programs and arise as trust-region subproblems in an algorithm for integer optimal control problems (IOCPs) with total variation penalization.
Optimal control problems are optimization problems which are constrained by ordinary or partial differential equations. IOCPs additionally restrict the control function to integer values. They are widely applicable modelling tools and their fields of application include the control of solar thermal climate systems, see \cite{burger2018solar}, aircraft trajectory planning, see \cite{Soler2016aircraft},  gas network control, see \cite{hante2017challenges}, and automotive control, see \cite{gerdts2005automobile}.

A popular approach to solve IOCPs is the combinatorial integral approximation due to optimality guarantees under conditions on the underlying differential equation, see \cite{sager2011combinatorial}. This approach however does not restrict the switching of the control function which can undermine the applicability, see \cite{hahn2021decomposition,bestehorn2021}, and their references for details.

One approach to decrease the switching is to include the total variation of the control function into the problem formulation.
\cite{buchheim2024parabolic} consider a bounded subset of functions with bounded variation on one-dimensional domains.
This allows for the construction of a branch-and-bound algorithm in order to solve parabolic optimal control problems with switching constraints, which can be extended to include additional combinatorial constraints. Instead of hard constraints, \cite{leyffer2022sequential} and \cite{marko2023integer} introduce a total variation penalty in the objective on one-dimensional domains.
An alternative way to account for switching costs is the combined minimization of a total number of switches and their switching times,
which is done by means of a proximal algorithm in \cite{de2019mixed}. Total variation penalties are also for smooth control functions and known as \emph{control volatility minimization}, see, e.g., \cite{loxton2020minimizing}.

In order to solve IOCPs with total variation penalization, \cite{leyffer2022sequential} propose a trust-region method for which the arising discretized subproblems are modelled as integer (linear) programs. The convergence and optimality results are given for an underlying one-dimensional domain. \cite{manns2023integer} extend the analysis to domains of dimension two and higher. The computational bottleneck of the approach is the computational demand of the underlying integer programs, currently solved without structure exploitation with off-the-shelf solvers. For the one-dimensional case, a shortest path or dynamic programming approach can be used to significantly accelerate the algorithm, see \cite{severitt2023efficient} and \cite{marko2023integer}. 
In this paper we turn to the two-dimensional case, analyzing the resulting subproblems and proposing a series of improvements to the integer programming formulation and its solution process with a standard solver in order to reduce the computational demand. 
We highlight that the discretized problems are interesting beyond the intended application in integer optimal control. Specifically, similar problems
can be found in image segmentation, see \cite{boykov2001,greig1989exact}, and multi-label optimization for Potts 
and Ising spin glass models, see \cite{Nieuwenhuis2013potts,de1995exact}, but \eqref{eq:trip} contains additional
 constraints such that it can be viewed as minimum $s$-$t$ cut problem with a knapsack-type constraint. 
\paragraph{Contribution}
We provide structural results for the underlying polyhedrons of an integer programming reformulation of \eqref{eq:trip}. We prove that the resulting problems are strongly NP-hard if the minimum bisection problem on subgraphs of the grid with an arbitrary number of holes is NP-hard. We extend our results in \cite{severitt2023efficient} and provide a conditional $p$-approximation for the integer programs.
We prove that the vertices of the underlying polyhedron can only attain non-integer values in connected parts of a corresponding graph. For the binary case, we show that every feasible point of the integer program is already a vertex of the polyhedron which is false for the non-binary case. We employ our findings in an integer programming solver-based solution process. We derive cutting planes which make use of this property of the polyhedron as well as an approach to improve primal points and a branching rule. We validate the improvements with respect to the run time on a numerical benchmark example.

\paragraph{Structure of the remainder} We introduce the problem class in Section \ref{sec: tr method} and briefly restate the trust-region algorithm from \cite{leyffer2022sequential}. We derive integer programs as reformulations of trust-region subproblems in Section \ref{sec: discrete problem}. We discuss NP-hardness for the subproblem in Subsection \ref{subsec: NP hard}. We introduce a Lagrangian relaxation in Subsection \ref{subsec: relax cap} and
analyze the connection to graph-based approaches, namely minimum $s$-$t$ cut problems, in Subsection \ref{subsec: connect graph}. Afterwards we state and prove the aforementioned property of the underlying polyhedron in Section \ref{sec: structure} which leads to cutting planes, primal points, and a branching rule in Section \ref{sec: IP based}. In Section \ref{sec: comp exp} we validate the proposed approaches 
computationally and discuss the results and how to gauge the computational demand of the integer program in Section \ref{sec: conclusion}. 

\paragraph{Notation} We abbreviate $[N] \coloneqq \{1, \ldots, N\}$. We introduce the notation $\lfloor x \rceil$ to represent the rounding of $x$ to the nearest integer value. In case of parity, $x$ is rounded up. The notation $\lceil x \rceil$ denotes rounding up to the nearest larger integer value while $\lfloor x \rfloor$ denotes rounding down.

In the paper we use the terminology ``polyhedron'' for which different definitions exist depending on the community. In our paper a polyhedron is the intersection of finitely many closed halfspaces.

\section{Trust-region method for IOCPs}
\label{sec: tr method}
In this section we introduce the motivating class of IOCPs as well as the trust-region algorithm employed to solve the IOCPs. We will discretize its subproblems in the next section and concern ourselves with the resulting discrete problem in the remainder of this paper.

Let $\alpha>0$ and $\Omega \subset \mathbb{R}^2$ be a rectangular domain. The IOCP reads
\begin{gather}
\begin{aligned}
    \min_{v \in L^2(\Omega)}\ &J(v) \coloneqq F(v)+\alpha \TV(v) \\
    \text{s.t. }\ \ & v(s) \in \Xi\coloneqq \{\xi_1, \ldots, \xi_m\}  \subset \mathbb{Z} \text{ for almost all (a.a.) } s \in \Omega.
    \label{eq:iocp}
\end{aligned}
\tag{IOCP}
\end{gather}
The function $F:L^2(\Omega) \mapsto \mathbb{R} $ is lower semicontinuous.
The term $\TV: L^1(\Omega) \to [0,\infty]$ denotes the total variation seminorm which models and penalizes the switching behaviour of the control function $v$. The set $\Xi$ contains all possible control values $\xi_1, \ldots, \xi_m$ and thus enforces integrality of the control function values. In this paper we assume that $\Xi$ is a contiguous set
of integers. 

The trust-region algorithm described in \cite{leyffer2022sequential} can be employed for problems of the form \eqref{eq:iocp}. The pseudo code is given in Algorithm \ref{alg: TR}. The algorithm consists of one outer and one inner loop. The inner loop solves a trust-region subproblem 
\begin{gather}
\begin{aligned}
\label{eq:tr}
pr(v, d) \coloneqq - \min_{d \in L^2(\Omega)} & (\nabla F(v), d)_{L^2(\Omega)} + \alpha \TV(v + d) - \alpha\TV(v) \\
\text{~~s.t.~~}
& v(s) + d(s) \in \Xi \text{ for a.a.\ } s \in \Omega,\\
& \|d\|_{L^1(\Omega)} \le \Delta
\end{aligned}
\tag{TR}
\end{gather}
to obtain a step inside the trust region. The negative objective value
of this partially linearized model over the trust-region is the \emph{predicted reduction}
and has the symbol $pr(v, d)$.
If the predicted reduction $pr$ is zero the algorithm terminates. The underlying optimality results and assumptions can be found in \cite{leyffer2022sequential} and \cite{manns2023integer}. If the predicted reduction is not zero it is checked if it exceeds a certain fraction of the actual reduction achieved by the solution to \eqref{eq:tr}. If yes, the step is accepted and the inner loop is terminated. Otherwise, the step is rejected and the inner loop begins anew with a reduced trust-region radius. The outer loop resets the trust-region radius to its initial value and triggers the inner loop again.

\begin{algorithm}
\caption{Sketch of the trust-region algorithm from \cite{leyffer2022sequential}}\label{alg:slip}
\textbf{Input: } feasible initial control $v^0 \in L^1(0,T)$ for \eqref{eq:iocp} (that is $v^0(t) \in \Xi$ for a.a.\ $t \in (0,T)$), reset
trust-region radius $\Delta^0 > 0$, 
acceptance ratio $\rho \in (0,1)$
\label{alg: TR}
\begin{algorithmic}[1]
\For{$n = 1,\ldots$}
\State $k \gets 0$, $\Delta^{n,0} \gets \Delta^0$
\Repeat
\State $d^{n,k} \gets$ minimizer of $\TR(v^{n-1},\Delta^{n,k})$
\Comment{Compute step.}
\If{$pr(v^{n-1}, d^{n,k}) = 0$}
\Comment{The predicted reduction is zero.}
\State Terminate with solution $v^{n-1}$.
\ElsIf{$J(v^{n - 1}) - J(v^{n-1} + d^{n,k}) < \rho pr(v^{n-1},d^{n,k})$}
\Comment{Reject step.}
\State $\Delta^{n,k+1} \gets \Delta^{n,k} / 2$, $k \gets k + 1$
\Else
\Comment{Accept step.}
\State $v^n \gets v^{n-1} + d^{n,k}$, $k \gets k + 1$
\EndIf
\Until{$J(v^{n - 1}) - J(v^{n-1} + d^{n,k}) \ge \rho pr(v^{n-1},d^{n,k})$}
\EndFor
\end{algorithmic}
\end{algorithm}
Because \eqref{eq:iocp} is posed in function space, we discretize the problem to solve it. An analysis of the discretization goes beyond the scope of this article and we use a uniform discretization here.

However, we note that discretizing the total 
variation and the controls with a uniform grid 
implies an anisotropic 
behavior of the solution that is governed by
the geometry of the grid 
cells. In particular, an anisotropic functional
is recovered in the limit when the mesh sizes
are driven to zero. The discretization dictates
the so-called Wulff shape of the functional,
see \cite{cristinelli2023conditional}.
We intend to integrate our approach in this work
into approximation schemes that
successively reduce the anisotropy of the
total variation functional
in the future.

We note that the choice of the $L^\infty$-norm instead of the $L^1$-norm also gives a polyhedral structure
of the trust-region constraint. The $L^\infty$-norm is not suitable for the trust-region algorithm on discrete-valued functions, however.
This is due to the fact that if two integer-valued functions coincide except for a set of arbitrarily
small measure, their $L^\infty$-norm is always bounded below by $1$ so that local optimality and thus stationarity cannot be defined and analyzed sensibly.

\section{The discretized trust-region subproblem and its relaxations}
\label{sec: discrete problem}
After a uniform discretization of the domain into $N \times M$ square cells, where $N,M \in \mathbb{N}$, the trust-region subproblems have the form
\begin{gather}
\begin{aligned}
    \min_{d \in \mathbb{Z}^{N\times M}}\quad & \sum_{i=1}^N \sum_{j=1}^M c_{i,j} d_{i,j} + \alpha  \sum_{j=1}^{M} \sum_{i=1}^{N-1}  \vert x_{i+1,j} + d_{i+1,j} - x_{i,j} - d_{i,j} \vert \\
    &+ \alpha  \sum_{i=1}^{N} \sum_{j=1}^{M-1}  \vert x_{i,j+1} + d_{i,j+1} - x_{i,j} - d_{i,j} \vert\\
    \text{s.t.}\quad
    & x_{i,j} + d_{i,j} \in \Xi \text{ for all } i \in [N], j \in [M]\quad\text{and}\quad \sum_{i=1}^N \sum_{j=1}^M  \vert  d_{i,j} \vert \le \Delta
\end{aligned}
\tag{TR-IP}
\end{gather}
with $c \in \mathbb{R}^{N \times M},\ x \in \Xi^{N \times M},\ \alpha>0,\ \Delta \in \mathbb{N}$. If $M=1$ or $N=1$ we call the problem one-dimensional, otherwise we refer to \eqref{eq:trip} as a two-dimensional problem, because the underlying structure can be viewed as an $N \times M$ grid, see Figure \ref{fig:grid}. The discretized trust-region constraint $\sum_{i=1}^N \sum_{j=1}^M  |d_{i,j}| \le \Delta$ is called the capacity constraint.
We have dropped a constant term corresponding to $\TV(v)$ in \eqref{eq:tr} from the objective as this does not affect our optimization. 

\begin{figure}
    \centering
	\resizebox{0.35 \textwidth}{!}{%

	\begin{tikzpicture}[node distance={65mm}, thick, main/.style = {draw, circle}] 
	\node[main, minimum size = 2.25 cm] (1) {\Huge $3,1$};
	\node[main, minimum size = 2.25 cm] (2) [right of=1] {\Huge $3,2$};
    \node[main, minimum size = 2.25 cm] (3) [right of=2] {\Huge $3,3$};

    \node[main, minimum size = 2.25 cm] (4) [below of=1]{\Huge $2,1$};
	\node[main, minimum size = 2.25 cm] (5) [below of=2] {\Huge $2,2$};
    \node[main, minimum size = 2.25 cm] (6) [below of=3] {\Huge $2,3$};
    
    \node[main, minimum size = 2.25 cm] (7) [below of=4]{\Huge $1,1$};
	\node[main, minimum size = 2.25 cm] (8) [below of=5] {\Huge $1,2$};
    \node[main, minimum size = 2.25 cm] (9) [below of=6] {\Huge $1,3$};

    	\node[main, minimum size = 2.25 cm] (10) [right of=3] {\Huge $3,4$};
	\node[main, minimum size = 2.25 cm] (11) [right of=6] {\Huge $2,4$};
    \node[main, minimum size = 2.25 cm] (12) [right of=9] {\Huge $1,4$};
    
    \draw (1) -- (2);
    \draw (1) -- (4);
    \draw (2) -- (3);
    \draw (2) -- (5);
    \draw (3) -- (6);
    \draw (4) -- (5);
    \draw (4) -- (7);
    \draw (5) -- (6);
    \draw (5) -- (8);
    \draw (6) -- (9);
     \draw (7) -- (8);
    \draw (8) -- (9);
    \draw (3) --(10);
    \draw(6) -- (11);
    \draw(9) --(12);
    \draw (10)--(11);
    \draw(11)--(12);
	\end{tikzpicture}
}

\caption{Example of an underlying grid with $N=3$ and $M=4$. For each node there exists a corresponding entry in the control $x+d$. For each edge there exists a corresponding absolute value term (horizontal: $\beta_{i,j}$, vertical: $\gamma_{i,j}$ in Subsection \ref{subsec:Ip form}) modeling the contribution of the control jump between neighbouring cells to the total variation.}
\label{fig:grid}
\end{figure}
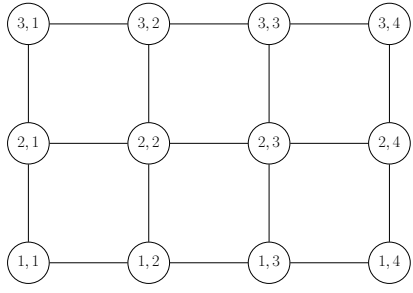

In Subsection \ref{subsec:Ip form} we will formulate \eqref{eq:trip} as an integer linear program and obtain the corresponding linear relaxation. Afterwards, we will motivate conjectures regarding the NP-hardness. In Subsection \ref{subsec: relax cap} we will introduce a Lagrangian relaxation and refer to it as \emph{the Lagrangian relaxation}. In Appendix \ref{subsec: dual decomp relax} an additional relaxation which we call \emph{the dual decomposition relaxation} can be found but is not included in the article itself because it did not prove useful in our preliminary computational experiments. For the Lagrangian relaxation we will prove an equivalence to the linear relaxation in Section \ref{sec: structure}.

\subsection{Integer programming formulation}
\label{subsec:Ip form}
By introducing auxiliary variables we are able to use linear inequalities to model the absolute values in the cost function and in the constraint. Thus \eqref{eq:trip} can be transformed into the integer linear program
\begin{gather}\label{eq:ip}
\begin{aligned}
    \min_{d,\delta,\beta,\gamma}\quad &  \sum_{i=1}^N \sum_{j=1}^M c_{i,j} d_{i,j} + \alpha  \left (\sum_{j=1}^{M} \sum_{i=1}^{N-1} \beta_{i,j} +  \sum_{i=1}^{N} \sum_{j=1}^{M-1} \gamma_{i,j} \right )\\
    \text{s.t.}\quad
    &  (d,\delta,\beta,\gamma) \in P_\Delta \text{ and }  d \in \mathbb{Z}^{N \times M},
\end{aligned}
\tag{IP}
\end{gather}
where $P_\Delta = \{ (d,\delta,\beta, \gamma) \in P: \ \sum_{i=1}^N \sum_{j=1}^M  \delta_{i,j} \le \Delta\} $ is the polyhedron obtained from the intersection of the capacity constraint and $P$ defined by 
\begin{align*}
    &(d, \delta,\beta,\gamma) \in P  
    &\logeq  \begin{cases}
         \min \Xi \leq x_{i,j} + d_{i,j} \leq \max \Xi &\text{ for all } i \in [N], j \in [M],\\
     -\beta_{i,j} \leq x_{i+1,j} + d_{i+1,j} - x_{i,j} - d_{i,j} \leq \beta_{i,j} &\text{ for all } i \in [N-1], j \in [M],\\
     -\gamma_{i,j} \leq x_{i,j+1} + d_{i,j+1} - x_{i,j} - d_{i,j} \leq \gamma_{i,j} &\text{ for all } i \in [N], j \in [M-1],\\
     -\delta_{i,j} \leq d_{i,j} \leq \delta_{i,j} &\text{ for all } i \in [N], j \in [M],
    \end{cases} 
\end{align*}
where we recall $x \in \Xi^{N\times M}$ so that there is always a feasible tuple with $d = 0$ and $\delta = 0$.
The corresponding linear programming relaxation reads
\begin{gather}\label{eq:lp}
\begin{aligned}
    \min_{d,\delta,\beta,\gamma}\quad &  \sum_{i=1}^N \sum_{j=1}^M c_{i,j} d_{i,j} + \alpha  \left (\sum_{j=1}^{M} \sum_{i=1}^{N-1} \beta_{i,j} +  \sum_{i=1}^{N} \sum_{j=1}^{M-1} \gamma_{i,j} \right )  \\
    \text{s.t.}\quad
    &  (d,\delta,\beta,\gamma) \in P_\Delta.
\end{aligned}
\tag{LP}
\end{gather}

\begin{remark}
\label{remark:sol construction}
    A feasible point $(d,\delta,\beta,\gamma)$ can only be optimal for \eqref{eq:ip} if $\beta_{i,j} = |x_{i+1,j}+d_{i+1,j}-x_{i,j}-d_{i,j}|$ and $\gamma_{i,j}=|x_{i,j+1}+d_{i,j+1}-x_{i,j}-d_{i,j}|$ because otherwise we could reduce the objective value by setting the values of $\beta$ and $\gamma$ to those absolute values. Furthermore, if $\delta_{i,j} > |d_{i,j}|$ we can always choose the minimal $\delta_{i,j}=|d_{i,j}|$ and remain feasible. Thus we can construct the corresponding $\delta,\beta$ and $\gamma$ from a given $d$.
    
    Consequently, if we say that $d^*$ is optimal or feasible we mean that the point $(d^*,\delta^*,\beta^*,\gamma^*)$ is optimal or feasible when $\delta^*,\beta^*$ and $\gamma^*$ are determined as above. 
\end{remark}

\subsection{NP-hardness}
\label{subsec: NP hard}
We now elaborate on the NP-hardness of the problem \eqref{eq:trip}. For the one-dimensional case the authors concerned themselves in \cite{severitt2023efficient} with the weighted problem 
\begin{gather}\label{eq:trip-w}
\begin{aligned}
    \min_{d \in \mathbb{Z}^{N}}\quad & \sum_{i=1}^N  c_{i} d_{i} +  \sum_{i=1}^{N-1} \tilde \alpha_{i}  \vert  x_{i+1} + d_{i+1} - x_{i} - d_{i} \vert \\
    \text{s.t.}\quad
    & x_{i} + d_{i} \in \Xi \text{ for all } i \in [N]\quad\text{and}\quad \sum_{i=1}^N  h_{i} \vert  d_{i} \vert \le \Delta,
\end{aligned}
\tag{wTR-IP}
\end{gather} where $\tilde \alpha_{i}, h_{i} \in \mathbb{R}_{\geq 0}.$ 
The NP-hardness for $\tilde \alpha \equiv 1$ and $h \in \mathbb{N}^N$ was proven by a reduction from knapsack. It is, however, solvable by a pseudo-polynomial algorithm.
We now motivate that we conjecture that in the two-dimensional case we can not find a pseudo-polynomial algorithm for \eqref{eq:trip} even if
$\Xi = \{0,1\}$, that is we conjecture that the problem \eqref{eq:trip} is strongly NP-hard. To this end we introduce the well-studied minimum bisection problem. 

\paragraph{Minimum bisection problem (MBP):} Given a graph $G(V,E)$, (MBP) seeks a partition into two sets $S, V\setminus S$ such that $|S|, |V \setminus S| \leq \left \lceil |V| / 2 \right \rceil$ which minimizes the cardinality of the set of cut edges $C \coloneqq \{(v,w) \in E | v \in  S,\ w \in V \setminus S  \}$. $|C|$ is called the bisection width.

(MBP) is NP-hard
for unit disc graphs, see \cite{DIAZ201783}. For planar graphs \cite{Papadimitriou1990TheBW} conjecture
that (MBP) remains NP-hard. There exists a polynomial reduction to the minimum bisection problem on subgraphs of the grid
$\mathbb{Z} \times \mathbb{Z}$ with an arbitrary 
number of holes (MBFH), see \cite{Papadimitriou1990TheBW}. So if (MBFH) is polynomially solvable so is (MBP) on planar 
graphs. We now give a polynomial reduction from (MBFH) to the binary \eqref{eq:trip} problem. Our reduction is
inspired by the ideas in \cite{Papadimitriou1990TheBW}.

In order to prove the reduction we need an auxiliary lemma, which is shown in Appendix \ref{sec: proof box formula}.
\begin{lemma}
\label{lemma: box formula}
    Let $G(V,A)$ be a $\tilde n \times \tilde m$ rectangular subgraph of the infinite grid $\mathbb{Z} \times \mathbb{Z}$. Then for a subset $U \subset G$ of size $K$ it holds that \[ | \partial U | \geq \min \{ \lfloor \sqrt{K} \rceil + \lceil \sqrt{K} \rceil, \tilde n, \tilde m, \lfloor \sqrt{\tilde n \tilde m - K} \rceil +\lceil \sqrt{\tilde n \tilde m - K} \rceil \}\] where  $\partial U =  \{ (v,w) \in A \ | \ v \in U, \ w \in G \setminus U \} $ is the set of cut edges. 
\end{lemma}

\begin{theorem}
    There exists a polynomial reduction from (MBFH) to \eqref{eq:trip} with a binary control and rational entries in $c$ bounded by $-9$ and $5$ which become integer values polynomial in the size of the grid graph if multiplied by $n^4$.
\end{theorem}
\begin{proof} {Proof}
    Without loss of generality we assume that the subgraph of the grid $G$ has an even number of nodes and more than $2$ to be nontrivial. Otherwise, we add a single node not connected to the rest of the graph and the following construction still works. In particular, $n \coloneqq |G| \geq 4$ is assumed. 
    
    We construct a new grid for which an optimal solution to \eqref{eq:trip} corresponds to an optimal solution to the bisection problem on the original graph.  
    We replace every node in the graph by a square of $n^4$ nodes and connect a square with a straight line to another square if the corresponding nodes in $G$ were connected by an edge. The straight line is connected to the middle of the sides of the squares on which sides of the nodes the edge in the original graph was adjacent to. We choose the node which lies closer to the bottom or the left of the square to decide the ties. The straight lines each contain one node with two adjacent edges which connect to one square respectively. This ensures that nodes of two different squares do not share an edge. Thus we obtain a connected graph with $n$ squares of size $n^4$. In total we connect the squares by less than $4 n$ straight lines each containing one node. We call the set of nodes in a square $V_1$ and the connecting nodes $V_2$. We now add nodes and the corresponding edges until we obtain a square grid with a size polynomially bounded in $n$ and in which every node of $V_1$ has exactly $4$ neighbouring nodes. We call the newly added nodes the set $V_3$. The constructed grid is visualized exemplarily in Figure \ref{fig: np proof}.
    We set $x \equiv 0.$ For a node $v$ in $V_1$ (nodes in a square) we set the cost $c(v) = -\frac{5}{n^4} - |\{(v,w)| w \in V_3 \}|$ while we set $c(v)= - |\{(v,w)| w \in V_3 \}| = -2$ for $v \in V_2$. For $v$ in $V_3$ we set $c(v) = 5$.  We set $\alpha=1$. This construction ensures that $d_v=0$ for all $v \in V_3$ as the increased objective, specifically $c_v d_v$, outweighs any possible reduction in the jumps regarding this node. Furthermore, the cost of a node $v \in V_2$ is constructed such that if $d_v=1$ the term $c_v d_v$ cancels out the jumps to adjacent nodes in $V_3$. The same holds true for the second term in the definition of the cost term for nodes in $V_1$. We choose $\Delta = \frac{n^5}{2}+ 4n$. This means that half the squares can be set to $1$ as well as all the nodes in $V_2$ which are not part of squares. When we say that we set a node $v$ to $1$ we mean that $d_v=1$.
    
    We now show that it is required for optimality for \eqref{eq:trip} for the feasible point to set half of the squares as well as the connecting nodes in $V_2$ to $1$.  These feasible points of \eqref{eq:trip} then corresponds to feasible points of the minimum bisection and minimizing the cost function of \eqref{eq:trip} is equivalent to minimizing the cut lines between the squares which correspond to minimizing the cut edges in the original subgraph. We start by showing that only feasible points of \eqref{eq:trip} can be optimal which set the nodes of $\frac{n}{2}$ entire squares to $1$. 
    
    We first show that it is suboptimal to set more than $0$ but less than $n^4-25$ nodes in a square to $1$. To this end we know that the formula $\min \{ \lfloor \sqrt{K} \rceil + \lceil \sqrt{K} \rceil, n^2,  \lfloor \sqrt{ n^4 - K} \rceil +\lceil \sqrt{ n^4 - K} \rceil \}$ provides a lower bound for the amount of jumps in the square if we set $K$ nodes to $1$. For $5 \leq K \leq n^4 -25$ the inequality 
    $ - \frac{5K}{n^4} + \min \{ \lfloor \sqrt{K} \rceil + \lceil \sqrt{K} \rceil, n^2, \lfloor \sqrt{ n^4 - K} \rceil +\lceil \sqrt{n^4 - K} \rceil \}  > 4  $ holds which is easy to check due to the concavity of the left side of the inequality. The left side is a lower bound for any feasible point which sets between $5$ and $n^4-25$ nodes in the square to $1$ while the right side is an upper bound for setting all nodes to $0$. We remind the reader that the second term $-|\{(v,w) | w \in V_3 \}|$ in the cost for a node $v$ in the square cancels out the jumps to adjacent nodes in $V_3$.  It follows that the cost occurring in the square is higher than a possible reduction in jumps of at most $4$ as there are at most $4$ nodes from $V_2$ connected to this square. Because these nodes in $V_2$ are at least $n^2-1$ nodes apart from each other the same argumentation that the cost in the square outweighs the cost reduction outside the square holds for $0<K<5$. We would need one separate component for each jump from a square node to a node in $V_2$ we want to eliminate which would instead produce at least two jumps in the square. Thus a feasible point can only be optimal if for all squares either all the nodes in the square are set to $0$ or at least $n^4-25$ nodes are set to $1$.

    We note that it takes a capacity of at least $(n^4-25) (\frac{n}{2}+1) = \frac{n^5}{2} + n^4 - \frac{25n}{2}-25$ for more than $\frac{n}{2}$ squares to set at least $n^4-25$ nodes to $1$ which exceeds the capacity bound of $\frac{n^5}{2} + 4n$ for $n \geq 4$. Thus a feasible point for \eqref{eq:trip} can only be optimal if exactly $\frac{n}{2}$ squares are set to $1$, because setting less squares to $1$ would be suboptimal as can be seen by the cost function (setting a whole square to $1$ reduces the total cost by at least $1$ even if all connecting nodes in $V_2$ are set to $0$).

    A feasible point which sets $\frac{n}{2}$ entire squares to $1$ minimizes the objective value inside the squares. If also the connecting nodes in $V_2$ in between these squares are set to $1$ the objective is further decreased outside of the squares. It does not effect optimality if a node in $V_2$ connecting a square which is set to $0$ and a square set to $1$ is itself set to $0$ or to $1$ due to the construction of the costs. Setting any other node, meaning a node in $V_1$ or $V_3$, to $1$ would increase the objective as previously shown. So we now need to choose the best feasible point from the set of feasible points which adhere to the described conditions in order to obtain an optimal solution for \eqref{eq:trip}. Thus the feasible point which sets exactly $\frac{n}{2}$ squares as well as the connecting nodes to $1$ is optimal if the amount of straight lines to the remaining squares is minimized which corresponds to the cut edges for the minimum bisection. The objective value is $-\frac{5n}{2} + C$ where $C$ is the minimal bisection of the original subgraph. Thus we obtain the desired value by adding $\frac{5n}{2}$ to the objective value of the optimal solution to \eqref{eq:trip}. 
\end{proof}
\begin{corollary}
    If (MBFH) is NP-hard and $NP\neq P$ then no (pseudo-) polynomial algorithm for \eqref{eq:trip} 
    and the two-dimensional variant of \eqref{eq:trip-w} exists.
\end{corollary}
\begin{proof} {Proof}
    The weights (including $\alpha=1$) are all integer values polynomial in $N$ if we multiply by $n^4$ and the trust-region radius is a polynomial in the size of the subgraph of the grid. 
\end{proof}

\begin{figure} 
    \centering
    \includegraphics[scale = 0.45]{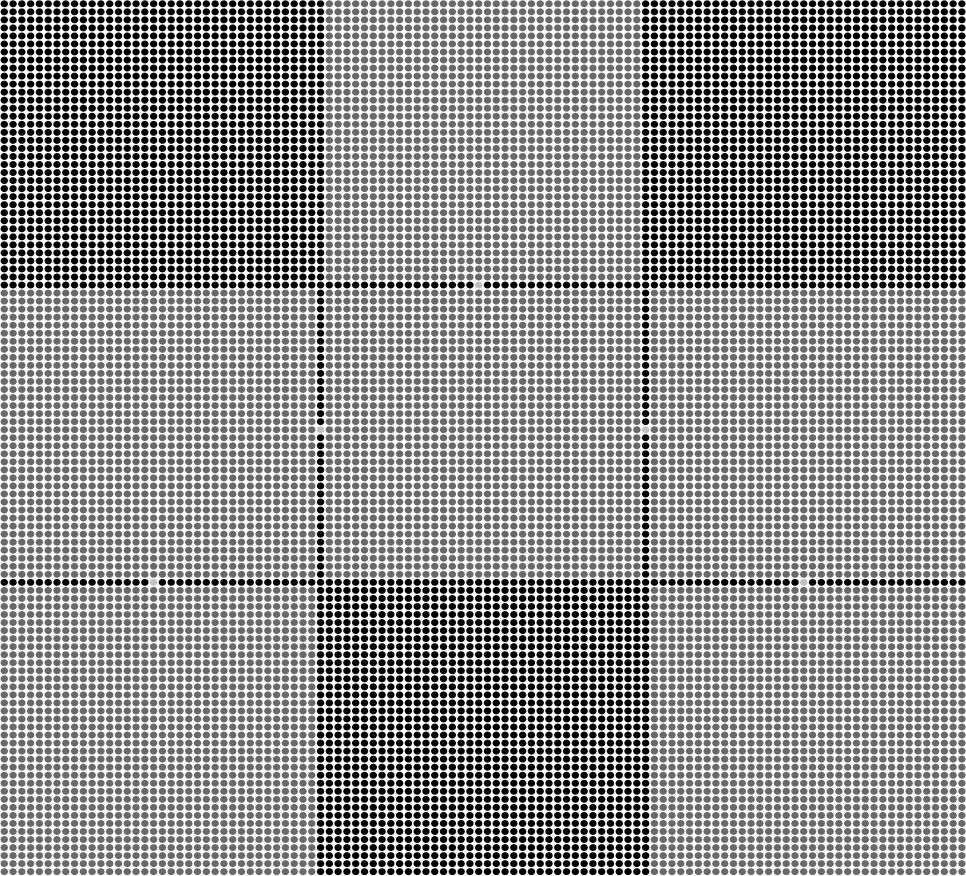}
    \caption{The original graph consisted of $6$ nodes. Each node becomes one square of $36 \times 36$ nodes in $V_1$. 
    	The set $V_3$ contains the black nodes.
    	The set $V_1$ is illustrated in dark grey. Two squares are connected by a node in $V_2$ if the corresponding nodes in the original graph were connected. 
    	The set $V_2$ contains the light grey nodes on the interfaces between these squares of nodes of $V_1$.}
    \label{fig: np proof}
\end{figure}

\subsection{Relaxation of the capacity constraint}
\label{subsec: relax cap}
We use a Lagrangian relaxation to move the capacity constraint as a penalty term into the objective. We will show that the resulting problem is polynomially solvable and provides a conditional $p$-approximation. As noted in the beginning of this section, we refer to this relaxation as the \emph{Lagrangian relaxation} for the remainder of the paper. We end up with the problem formulation
\begin{gather}\label{eq:lr}
\begin{aligned}
    \max_{\mu \geq 0} \min_{d,\delta,\beta,\gamma}\quad & \sum_{i=1}^N \sum_{j=1}^M c_{i,j} d_{i,j} + \alpha  \left (\sum_{j=1}^{M} \sum_{i=1}^{N-1} \beta_{i,j} +  \sum_{i=1}^{N} \sum_{j=1}^{M-1} \gamma_{i,j} \right )  + \mu \left ( \sum_{i=1}^N \sum_{j=1}^M  \delta_{i,j} - \Delta \right ) \\
    \text{s.t.}\quad
    & (d,\delta,\beta,\gamma) \in P \text{ and } d \in \mathbb{Z}^{N \times M}.
\end{aligned}
\tag{LR-$\Delta$}
\end{gather}
The parameter $\mu \geq 0$ penalizes the capacity consumption and we can ensure that an optimal solution adheres to the capacity constraint by choosing $\mu$ large enough. We will see in Subsection \ref{subsec: connect graph} that for a fixed $\mu$ the inner minimization problem can be solved in polynomial time and the optimal $\mu$ can be determined with a binary search.
An optimal solution to this relaxation provides, based on the used capacity, a $p$-approximation for the problem 

\begin{gather}\label{eq:trip-og}
\begin{aligned}
    \min_{d \in \mathbb{Z}^{N\times M}}\quad & \sum_{i=1}^N \sum_{j=1}^M c_{i,j} d_{i,j} + \alpha  \sum_{j=1}^{M} \sum_{i=1}^{N-1}  \vert x_{i+1,j} + d_{i+1,j} - x_{i,j} - d_{i,j} \vert \\
    &+ \alpha  \sum_{i=1}^{N} \sum_{j=1}^{M-1}  \vert x_{i,j+1} + d_{i,j+1} - x_{i,j} - d_{i,j} \vert\\
    &- \alpha  \sum_{j=1}^{M} \sum_{i=1}^{N-1}  \vert x_{i+1,j}   - x_{i,j} \vert 
    - \alpha  \sum_{i=1}^{N} \sum_{j=1}^{M-1}  \vert x_{i,j+1} - x_{i,j}  \vert\\
    \text{s.t.}\quad
    & x_{i,j} + d_{i,j} \in \Xi \text{ for all } i \in [N], j \in [M]\quad\text{and}\quad \sum_{i=1}^N \sum_{j=1}^M  \vert  d_{i,j} \vert \le \Delta,
\end{aligned}
\tag{OG-TR-IP}
\end{gather}
which is the problem \eqref{eq:trip} before dropping the last two constant terms corresponding to $-\alpha \TV(v)$ in \eqref{eq:tr} from the objective. A $p$-approximation guarantee for this problem, which is the negative predicted reduction in the trust-region algorithm, allows to use feasible points satisfying the $p$-approximation in similar ways as Cauchy points instead of optimal points in trust-region algorithms while retaining the convergence properties. For sake of clarity we now define $ \omega(d) \coloneqq \alpha  \sum_{j=1}^{M} \sum_{i=1}^{N-1}  \vert x_{i+1,j} + d_{i+1,j} - x_{i,j} - d_{i,j} \vert + \alpha  \sum_{i=1}^{N} \sum_{j=1}^{M-1}  \vert x_{i,j+1} + d_{i,j+1} - x_{i,j} - d_{i,j} \vert $ and $\bar C \coloneqq \alpha  \sum_{j=1}^{M} \sum_{i=1}^{N-1}  \vert x_{i+1,j}   - x_{i,j} \vert  + \alpha  \sum_{i=1}^{N} \sum_{j=1}^{M-1}  \vert x_{i,j+1} - x_{i,j} \vert .$

\begin{theorem}[Conditional $p$-approximation]
\label{theorem: cond p}
    Let $( \bar d, \bar \delta, \bar \beta, \bar \gamma,  \bar \mu)$ be optimal for \eqref{eq:lr} and let $d^*$ be optimal for \eqref{eq:trip-og}. If $\Delta \geq \sum_{i=1}^N \sum_{j=1}^M \bar \delta_{i,j} \geq p \Delta$ for $p \in (0,1]$ then it holds that 

    \[ c^T \bar d + \omega(\bar d) - \bar C \leq p(c^T d^* + \omega(d^*) - \bar C)       . \]
\end{theorem}

\begin{proof} {Proof}
Let $(\bar{d},\bar{\delta},\bar{\beta},\bar{\gamma},\bar{\mu})$ be optimal for \eqref{eq:lr} and we note that $p \in (0,1]$ gives that
$\bar{d}$ is also feasible for \eqref{eq:trip-og}. 
We first prove the case $p=1$.
The inequality $c^T \bar{d} + w(\bar{d})-\bar C + \bar{\mu} (\sum_{i=1}^N \sum_{j=1}^M  \bar{\delta}_{i,j} - \Delta) \leq c^Td+w(d)-\bar C $
holds for every $d$ feasible for \eqref{eq:trip-og}.
If $\sum_{i=1}^N \sum_{j=1}^M \bar{\delta}_{i,j} = \Delta$ then the objective values of \eqref{eq:trip-og} and \eqref{eq:lr} coincide.

We now turn to the case $p \in (0,1).$
    We prove this result by way of contradiction and assume
     \begin{align}
          c^T \bar d + \omega(\bar d) - \bar C > p(c^T d^* + \omega(d^*) - \bar C)
          ,
          \label{eq: approx: contra}
     \end{align}
     where we note that both the left and right hand side of \eqref{eq: approx: contra} are non-positive because $d = 0$ is feasible
     for both optimization problems.
     Because $( \bar d, \bar \delta, \bar \beta, \bar \gamma, \bar \mu)$ is optimal for \eqref{eq:lr} it holds that 
     \begin{align}
          c^T \bar d + \omega(\bar d) - \bar C + \bar \mu \sum_{i=1}^N \sum_{j=1}^M \bar \delta_{i,j}   \leq c^T d^* + \omega(d^*) - \bar C  + \bar \mu \sum_{i=1}^N \sum_{j=1}^M  \delta^*_{i,j}.
          \label{eq: approx: lagrangeopt}
     \end{align}
     Using the two inequalities \eqref{eq: approx: contra} and \eqref{eq: approx: lagrangeopt} as well as $ \sum_{i=1}^N \sum_{j=1}^M \bar \delta_{i,j} \geq p \Delta$ and $\sum_{i=1}^N \sum_{j=1}^M  \delta^*_{i,j} \leq \Delta,$ we obtain that 
     $ (1-p) (c^Td^* + \omega(d^*) - \bar C) > \sum_{i=1}^N \sum_{j=1}^M  (\bar \delta_{i,j} - \delta^*_{i,j}) \geq - (1-p) \bar \mu \Delta. $
     After multiplication with $\frac{p}{1-p}>0$ we obtain that
     \begin{align}
         p (c^Td^* + \omega(d^*) - \bar C) > - \bar \mu p \Delta.
         \label{eq:approx1}
     \end{align}
The point $\tilde d \equiv 0$ is feasible for \eqref{eq:lr} with an objective value of $\bar C - \bar \mu\Delta $ and thus
$c^T \bar d + \omega(\bar d) - \bar C + \bar \mu (\sum_{i=1}^N \sum_{j=1}^M \bar \delta_{i,j} - \Delta) \leq  - \bar \mu \Delta$
holds.
From $\sum_{i=1}^N \sum_{j=1}^M \bar \delta_{i,j} - \Delta \geq (p-1) \Delta$ it follows that 
\begin{align}
    c^T \bar d + \omega(\bar d) - \bar C   \leq  - \bar \mu p \Delta.
    \label{eq:approx2}
\end{align}  
Combining \eqref{eq:approx1} and \eqref{eq:approx2} we obtain the contradiction $c^T \bar d + \omega(\bar d) - \bar C  < p (c^Td^* + \omega(d^*) - \bar C)$.
\end{proof}

\begin{remark}
Theorem \ref{theorem: cond p} provides a $p$-approximation if the trust region is used at least by a fraction $p$ at the optimal solution to
\eqref{eq:lr}. Intuitively, this holds when the LP-relaxation
is already integer-valued in many entries of $d$ and thus close to the solution to
\eqref{eq:lr}. Note that they always coincide in the integer entries of the
LP-relaxation, which follows from the proof of Theorem \ref{thm:equivalenceLrLp} below.
Since \eqref{eq:lr} can generally be solved much faster than
\eqref{eq:trip} or \eqref{eq:trip-og} respectively, this can be used to accept
steps faster within the superordinate trust-region algorithm although early
numerical experiments suggest that the possible improvements are not very significant.
\end{remark}

The following corollary restates the case $p=1$ and was proven in Proposition 14 in \cite{severitt2023efficient} for the one-dimensional case, now extended to the two-dimensional case.

\begin{corollary}
Let $(d^*,\delta^*,\beta^*,\gamma^*,\mu^*)$ be optimal for \eqref{eq:lr} and $\sum_{i=1}^N \sum_{j=1}^M  \delta^*_{i,j} = \Delta$.
Then $(d^*,\delta^*,\beta^*,\gamma^*)$ is optimal for \eqref{eq:ip}.
\end{corollary}

\subsection{Connection to efficient graph algorithms}
\label{subsec: connect graph}
In one dimension ($M=1$ or $N=1$) an optimal solution for \eqref{eq:trip} can be computed by means of a reformulation as a shortest-path problem, see \cite{severitt2023efficient}. This approach has pseudo-polynomial complexity because the size of the graph grows linear in the input value $\Delta$, but is polynomial if $\Xi$ is a fixed contiguous set of integers, where $\Delta$ can be bounded by $|\Xi|N$.

This approach can not be applied in two dimensions by traversing the underlying grid in a one-dimensional sequence, as either half of the terms modeling the total variation would have to be ignored or the size of the graph would have to grow exponentially to encode the necessary information of at least the previous $\min\{N,M\}$ graph layers to guarantee optimality, see Figure \ref{fig:traversal}.

\begin{figure}
    \centering
	\resizebox{0.35 \textwidth}{!}{%
	
	\begin{tikzpicture}[node distance={65mm}, thick, main/.style = {draw, circle}]  
	\node[main, minimum size = 2.25 cm] (1) {\Huge $1$};
	\node[main, minimum size = 2.25 cm] (2) [right of=1] {\Huge $4$};
    \node[main, minimum size = 2.25 cm] (3) [right of=2] {\Huge $7$};

    \node[main, minimum size = 2.25 cm] (4) [below of=1]{\Huge $2$};
	\node[main, minimum size = 2.25 cm] (5) [below of=2] {\Huge $5$};
    \node[main, minimum size = 2.25 cm] (6) [below of=3] {\Huge $8$};
    
    \node[main, minimum size = 2.25 cm] (7) [below of=4]{\Huge $3$};
	\node[main, minimum size = 2.25 cm] (8) [below of=5] {\Huge $6$};
    \node[main, minimum size = 2.25 cm] (9) [below of=6] {\Huge $9$};

    	\node[main, minimum size = 2.25 cm] (10) [right of=3] {\Huge $10$};
	\node[main, minimum size = 2.25 cm] (11) [right of=6] {\Huge $11$};
    \node[main, minimum size = 2.25 cm] (12) [right of=9] {\Huge $12$};
    
    \draw[red] (1) -- (2);
    \draw (1) -- (4);
    \draw (2) -- (3);
    \draw (2) -- (5);
    \draw (3) -- (6);
    \draw (4) -- (5);
    \draw (4) -- (7);
    \draw (5) -- (6);
    \draw (5) -- (8);
    \draw (6) -- (9);
     \draw (7) -- (8);
    \draw (8) -- (9);
    \draw (3) --(10);
    \draw(6) -- (11);
    \draw(9) --(12);
    \draw (10)--(11);
    \draw(11)--(12);
	\end{tikzpicture}
}
\caption{We use a shortest-path approach to determine the optimal control in each node in the order $1$ to $12$. To guarantee optimality we need to take the control value in node $1$ into account when choosing the control value in node $4$ in order to accurately model the red edge between the nodes. Thus we need to encode the last $\min\{N,M\}=3$ control values in a graph construction which means the size grows exponentially in $\min\{N,M\}$. If all row edges are ignored, the resulting problem can be solved by a pseudo-polynomial algorithm in the same manner as the one-dimensional \eqref{eq:trip}.}
\label{fig:traversal}
\end{figure}
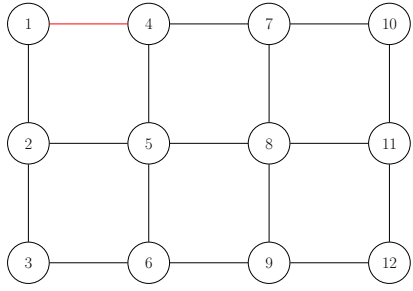
Instead, a capacity-constrained minimum $s$-$t$ cut seems like a better fit for the two-dimensional case because dropping the capacity constraint gives a polynomially solvable minimum cut problem.
One can reformulate the problem \eqref{eq:trip} as a capacity-constrained minimum $s$-$t$ cut problem on a graph $G = (V,A)$, which searches for a minimum $s$-$t$ cut $C$ with regards to a weight function $w_1: A \rightarrow \mathbb{R}_{\geq 0}$ and adheres to a capacity constraint $w_2(C) \leq \Delta$ with a capacity consumption function $w_2: A \rightarrow \mathbb{R}_{\geq 0}$. The graph construction is derived from \cite{Veksler99} which tackles a similar problem without a capacity constraint which in our case is modelled by $w_2$. The idea of using a minimum cut approach for energy minimization is common in image segmentation, see e.g.\ \cite{boykov2001}. The construction of the minimum $s$-$t$ cut problem after dropping
the capacity constraint mirrors the one in \cite{Veksler99}.
This also shows that the Lagrangian relaxation problem \eqref{eq:lr} is polynomially solvable for a fixed $\mu$.
Capacity-constrained global minimum cuts can be calculated efficiently as  bicriteria minimum cuts, see Theorem 2.4 in \cite{armon2006} . The bicriteria $s$-$t$ minimum cut problem is NP-hard in general, see Theorem 6 in \cite{Papadimitriou892068}, but their proof does not extend to grid graphs. 

\section{Structure of the polyhedron $P$}
\label{sec: structure}
In this section we will analyze the polyhedron described by the inequalities of \eqref{eq:ip} and \eqref{eq:lp} as well as the relationship between the two relaxations \eqref{eq:lp} and \eqref{eq:lr}.
The underlying polyhedron has a special structure which will later lead to valid cutting planes for \eqref{eq:ip}. To describe this special characteristic of the polyhedron we need the following definition.

\begin{definition}
We call the set $\Omega \coloneqq \{(i,j) | i,j \in [N] \}$ the \textbf{set of all index pairs}. Two index pairs $(i,j),(k,l) \in \Omega $ are \textbf{adjacent} if the index pair $(k,l)$ is equal to one of the index pairs $(i+1,j),(i-1,j),(i,j+1), (i,j-1)$. We say that an index pair $(k,l)$ is \textbf{adjacent to a subset of index pairs} if it is adjacent to at least one index pair in the subset but is not part of the subset itself.

A subset of index pairs $M \subset \Omega$ is a \textbf{connected component} if it contains only one element or for every index pair $(i,j) \in M$ at least one adjacent index pair is also in $M$ and for any two index pairs $(i,j),(k,l) \in M$ the condition $d_{i,j}+x_{i,j}=d_{k,l}+x_{k,l}$ holds. We refer to $d_{i,j}+x_{i,j}$ as the \textbf{value of the index pair} $(i,j)$. A connected component is \textbf{maximal} if for every $(i,j) \in M $ the condition $d_{i,j}+x_{i,j} \neq d_{o,p}+x_{o,p}$ for $(o,p) \in \{(i+1,j),(i-1,j),(i,j+1), (i,j-1)\} \cap (\Omega \setminus M)$ holds.
\end{definition}

Our goal is to show that all vertices of the polyhedron only take values in $\Xi$ for $x+d$ except in at most one maximal connected component. We will 
use this property to
construct a solution to \eqref{eq:lr} from a solution to \eqref{eq:lp} and vice versa. In Subsection \ref{subsec:cuts} we will construct cuts based on this property of the polyhedron. Before stating and proving the result,
we need an auxiliary lemma.

\begin{lemma}
    For $c_1,c_2>0$ and $k,\ell \in \mathbb{R}$ the problem 
\begin{gather}
    \min_{s_1,s_2}\enskip 0 \quad\textup{s.t.}\quad
    ks_1+\ell s_2=0 \textup{ and } -c_1 \leq s_1 \leq c_1 \textup{ and } -c_2 \leq s_2 \leq c_2
\end{gather} 
has a non-trivial optimal solution $s=(s_1,s_2)$ for which $-s$ is also an optimal solution.
\end{lemma}
\begin{proof} {Proof}
If $k=0$ then a possible optimal solution is $s=(c_1,0)$. It also holds that $(-c_1,0)$ is an optimal solution. The same argument holds for $\ell=0$ and $s=(0,c_2)$. Let $k \neq 0$ and $ \ell \neq 0 $ then we can choose $s_1=-\frac{\ell}{k} s_2$ and $s_2= \min\{c_2, \vert \frac{k}{\ell} \vert c_1\}$ which is feasible for the problem and thus an optimal solution. $-s$ is also an optimal solution because the bounds for $s_1,s_2$ are symmetric. 
\end{proof}

\begin{theorem}
    Every vertex solution to \eqref{eq:lp} has at most one maximal connected component with fractional values.
\end{theorem}
\begin{proof} {Proof}
Let $(d^*,\delta^*, \beta^*, \gamma^*)$ be optimal for \eqref{eq:lp} with two fractional connected components $M_1,M_2$  and for all $(i,j) \in \Omega$ the equalities $|d_{i,j}^*| = \delta_{i,j}^*$, $|d_{i+1,j}^*+x_{i+1,j}- d_{i,j}^*-x_{i,j}| = \beta_{i,j}^*$ and $|d_{i,j+1}^*+x_{i,j+1}- d_{i,j}^*-x_{i,j}| = \gamma_{i,j}^*$ hold (this is wlog because for every feasible point of the relaxation, we can find a feasible point of this form, which would have a better objective value in the case that the second or third condition was no met by the original point, see Remark \ref{remark:sol construction}). Because the set $\Omega$ contains a finite amount of elements, all connected components are finite. Thus we can assume that both connected components are maximal (otherwise we add the missing index pairs).

All index pairs adjacent to one of the connected components have a strictly larger or smaller value. Let $v_1$ be the value of the index pairs in $M_1$ and $v_2$ be the value of the index pairs in $M_2$. Because $\Xi$ is a finite set we can find $\ell_1 \coloneqq \max \{ \xi \in \Xi | \xi < v_1\}$ and $u_1 \coloneqq \min \{ \xi \in \Xi | \xi > v_1\}$ for $M_1$ and $\ell_2 \coloneqq \max \{ \xi \in \Xi | \xi < v_2\}$ and $u_2 \coloneqq \min \{ \xi \in \Xi | \xi > v_2\}$ for $M_2$.

We now have to distinguish between two cases. For the first case we assume that $M_1$ and $M_2$ are adjacent to each other. Because both components are maximal it follows that $v_1 \neq v_2$.
We introduce 
\begin{align}
\label{cond1}
t_k(i,j) = \begin{cases}
s_k &\text{ if } (i,j) \in M_k,\\
0 &\text{ otherwise }
\end{cases}
\end{align}
with $s_k \in [\max\{\ell_k-v_k,v_k-u_k, - \frac{1}{2}|v_1-v_2|\},\min\{u_k - v_k,v_k-\ell_k,\frac{1}{2}|v_1-v_2|\}]$ with $k \in {1,2}$.  

We set $t \coloneqq t_1+t_2$.
Because $x_{i,j} \in \Xi$ holds for all $(i,j) \in \Omega$, our choices for the lower bounds $\ell_1, \ell_2$ and for the upper bounds 
$u_1,u_2$ ensure that the condition $\sign(d_{i,j}+t(i,j))= \sign(d_{i,j}-t(i,j)) =\sign(d_{i,j}) $  holds, which follows from $\max\{\ell_k-v_k,v_k-u_k\} \leq s_k \leq \min\{u_k - v_k,v_k-\ell_k\}$ and where we assume that the condition also holds if one  side is $0$. We obtain that
\begin{align*}
    \sum_{(i,j) \in M_1} (\vert d_{i,j} +t(i,j) \vert - \vert d_{i,j} \vert) + \sum_{(i,j) \in M_2} (\vert d_{i,j} +t(i,j) \vert - \vert d_{i,j} \vert)  \\
= \sum_{\substack{(i,j) \in M_1, \\ d_{i,j}>0}} s_1+ \sum_{\substack{(i,j) \in M_1, \\ d_{i,j}<0}} -s_1 + \sum_{\substack{(i,j) \in M_2, \\ d_{i,j}>0}} s_2 + \sum_{\substack{(i,j) \in M_2, \\ d_{i,j}<0}} -s_2
\end{align*} 
and 
\begin{align*}
    \sum_{(i,j) \in M_1} (\vert d_{i,j} -t(i,j) \vert - \vert d_{i,j} \vert) + \sum_{(i,j) \in M_2} (\vert d_{i,j} -t(i,j) \vert - \vert d_{i,j} \vert) \\
= \sum_{\substack{(i,j) \in M_1, \\ d_{i,j}>0}} -s_1+ \sum_{\substack{(i,j) \in M_1, \\ d_{i,j}<0}} s_1 + \sum_{\substack{(i,j) \in M_2, \\ d_{i,j}>0}} -s_2 + \sum_{\substack{(i,j) \in M_2, \\ d_{i,j}<0}} s_2 
\end{align*} 
hold for any $t$ constructed as above.
If the equality 
\begin{align}
\label{cond3}
\sum_{\substack{(i,j) \in M_1, \\ d_{i,j}>0}} -s_1+ \sum_{\substack{(i,j) \in M_1, \\ d_{i,j}<0}} s_1 + \sum_{\substack{(i,j) \in M_2, \\ d_{i,j}>0}} -s_2 + \sum_{\substack{(i,j) \in M_2, \\ d_{i,j}<0}} s_2 = 0
\end{align} 
holds, it guarantees that the two points with $\bar d=d^*+t$ and $\tilde d=d^*-t$ have the same capacity consumption as the point with $d^*$, if we set $|\bar d_{i,j}| = \bar \delta_{i,j}$ and  $|\tilde d_{i,j}| = \tilde \delta_{i,j}$ for all $i,j$, which means that they adhere to the capacity constraint of (LP). This also shows that the assumption $\delta^*_{i,j} = |d^*_{i,j}|$ was not a restriction, because we could also choose larger $\bar \delta_{i,j}$ and $\tilde \delta_{i,j}$ if $\delta^*_{i,j}$ were to be larger than $|d^*_{i,j}|$. In combination with $t = t_1 + t_2$, the conditions \eqref{cond1} and \eqref{cond3} describe the linear program
\begin{gather}
\begin{aligned}
    \min_{s_1,s_2}\quad & 0 \\
    \text{s.t.}\quad
    & \sum_{\substack{(i,j) \in M_1, \\ d_{i,j}>0}} s_1+ \sum_{\substack{(i,j) \in M_1, \\ d_{i,j}<0}} -s_1 + \sum_{\substack{(i,j) \in M_2, \\ d_{i,j}>0}} s_2 + \sum_{\substack{(i,j) \in M_2, \\ d_{i,j}<0}} -s_2 = 0, \\
    & \max\{\ell_1-v_1,v_1-u_1, -\tfrac{1}{2}|v_1-v_2|\}\leq s_1 \leq \min\{u_1 - v_1,v_1-\ell_1, \tfrac{1}{2}|v_1-v_2|\}, \\
    & \max\{\ell_2-v_2,v_2-u_2, -\tfrac{1}{2}|v_1-v_2|\}\leq s_2 \leq \min\{u_2 - v_2,v_2-\ell_2, \tfrac{1}{2}|v_1-v_2|\}.
\end{aligned}
\end{gather}

From our auxiliary lemma we obtain that there exists $s =(s_1,s_2)\neq 0$ for which $s$ and $-s$ solve the linear program. We choose $s_1$ and $s_2$ for $t$ accordingly. We set  $\bar\beta_{i,j} = |\bar d_{i+1,j}+x_{i+1,j}- \bar d_{i,j}-x_{i,j}| ,$   $\bar \gamma_{i,j} = |\bar d_{i,j+1}+x_{i,j+1}- \bar d_{i,j}-x_{i,j}|,$ $ \tilde \beta_{i,j} = |\tilde d_{i+1,j}+x_{i+1,j}- \tilde d_{i,j}-x_{i,j}|$, and $\tilde \gamma_{i,j} = |\tilde d_{i,j+1}+x_{i,j+1}- \tilde d_{i,j}-x_{i,j}|$. We now show
\begin{align*}
    \sign(\bar d_{i+1,j}+x_{i+1,j}- \bar d_{i,j}-x_{i,j}) &= \sign( d_{i+1,j}^*+x_{i+1,j}-  d_{i,j}^*-x_{i,j}) \\ &= \sign(\tilde d_{i+1,j}+x_{i+1,j}- \tilde d_{i,j}-x_{i,j}),
\end{align*}
where we assume that equality also holds if one side is $0$. It then follows that $\frac{1}{2} \bar \beta + \frac{1}{2} \tilde \beta = \beta^*$. 
If $(i,j), (i+1,j) \in \Omega$ are in the same or in neither fractional component, then it follows from $t(i+1,j)=t(i,j)$ that
\begin{align*}
d_{i+1,j}^*+x_{i+1,j}-  d_{i,j}^*-x_{i,j} 
&= d_{i+1,j}^* + t(i+1,j) +x_{i+1,j}-  d_{i,j}^*-t(i,j)-x_{i,j} \\
&= d_{i+1,j}^* - t(i+1,j) +x_{i+1,j}-  d_{i,j}^*+t(i,j)-x_{i,j}.
\end{align*}

If $(i,j) \in M_1$ and $(i+1,j)  \in \Omega \setminus (M_1 \cup M_2)$ then $d_{i,j}+t(i,j)$ and $d_{i,j}-t(i,j)$ are bounded by the closest integer values and thus the sign remains the same because $d_{i+1,j}+x_{i+1,j}$ is an integer and $t(i+1,j)=0$.

If $(i,j) \in M_1$ and $(i+1,j) \in M_2$ then the condition $-\frac{1}{2}|v_1-v_2|\leq s_k \leq \frac{1}{2}|v_1-v_2|$ ensures that the sign does not change because $d_{i,j} + x_{i,j} > d_{i+1,j}+x_{i+1,j}$ implies
\[d_{i,j}+t(i,j) + x_{i,j} =v_1+s_1 \geq v_2+s_2 = d_{i+1,j}+t(i+1,j)+x_{i+1,j}\] and vice versa.

Thus in all possible cases the signs remain the same.
By the same argumentation $\frac{1}{2} \bar \gamma + \frac{1}{2} \tilde \gamma = \gamma^*$ holds.
Thus $(\bar d, \bar \delta, \bar \beta,\bar \gamma)$ and $(\tilde d, \tilde \delta, \tilde \beta, \tilde \gamma)$ are both feasible for (LP) and $(d^*,\delta^*, \beta^*, \gamma^*)$ is a convex combination of the two, which means that it is not a vertex. 

In the second case in which the fractional components are not adjacent, the proof remains the same except that the condition $-\frac{1}{2}|v_1-v_2| \leq s_k \leq \frac{1}{2}|v_1-v_2|$ is dropped.

It follows that no feasible point with more than one maximal connected component with fractional values is a vertex of the polyhedron. 
\end{proof}

\begin{theorem}\label{thm:capacity_consumption}
    Every vertex solution $(d,\delta,\beta,\gamma)$ of \eqref{eq:lp} fulfills the capacity constraint with equality or $x+d \in \Xi^{N \times N}$.
\end{theorem}
\begin{proof} {Proof}
Let $(d^*,\delta^*, \beta^*, \gamma^*)$ be optimal for \eqref{eq:lp} with a fractional connected component for which the capacity constraint is inactive. We define $\ell \coloneqq \max \{ \xi \in \Xi | \xi < v\}$ and $u \coloneqq \min \{ \xi \in \Xi | \xi > v\}$ where $v$ is the fractional value of the connected component $M$. We define  \[t(i,j) = \Bigg\{\begin{matrix}
s &\text{ if } (i,j) \in M,\\
0 &\text{ otherwise }
\end{matrix} \]
with $s \in [\max(\ell-v, v-u),\min(u - v, v-\ell)]$. Furthermore, we ensure $\sum_{j=1}^N \sum_{i=1}^N  \vert  d_{i,j}+t(i,j) \vert \le \Delta$ and $\sum_{j=1}^N \sum_{i=1}^N  \vert  d_{i,j}-t(i,j) \vert \le \Delta$. Because $\sum_{j=1}^N \sum_{i=1}^N  \vert  d^*_{i,j} \vert < \Delta$ it follows from $\ell-v<0<u-v$ that we can find a $t \neq 0$ that satisfies all conditions with the arguments from above. We obtain that $\bar d = d^*+t$ and $\tilde d = d^*-t$ satisfy the capacity constraint. We set $\bar \delta_{i,j} 
= |\bar d_{i,j}|$, $\bar \beta_{i,j} = |\bar d_{i,j} + x_{i,j} -\bar d_{i+1,j} - x_{i+1,j}| $ and $\bar \gamma_{i,j} = |\bar d_{i,j+1} +x_{i,j+1} - \bar d_{i,j}-x_{i,j}|$ and define $\tilde \delta$, $\tilde \beta$ and $\tilde \gamma$ in analogously. Then $(\bar d, \bar \delta, \bar \beta, \bar \gamma)$ and $(\tilde d, \tilde \delta, \tilde \beta, \tilde \gamma)$ are feasible for \eqref{eq:lp} which means the convex combination $(d^*,\delta^*, \beta^*, \gamma^*)$  is not a vertex of the polyhedron of \eqref{eq:lp}.  It follows that no feasible point with at least one maximal connected component with fractional values and inactive capacity constraint is a vertex. 
\end{proof}

We will use this structure in Subsection \ref{subsec:cuts} to obtain valid cuts for our linear relaxation \eqref{eq:lp}. The previous result directly implies the connection between the two relaxations $\eqref{eq:lp}$ and $\eqref{eq:lr}$.
\begin{corollary}
        Every vertex solution $(d,\delta,\beta,\gamma, \mu) $ of \eqref{eq:lr} fulfills $x+d \in \Xi^{N \times N}$ even when the integrality constraint is dropped.
\end{corollary}
\begin{proof} {Proof}
Regardless of the integrality constraint, every feasible point of \eqref{eq:lr} has a capacity consumption of at most $N^2(\max \Xi-\min \Xi)$. Therefore adding the constraint $\sum_{j=1}^N \sum_{i=1}^N  \delta_{i,j} \le N^2(\max \Xi -\min \Xi )+1$ to $\eqref{eq:lr}$ without the integrality constraint does not change the 
feasible set. Theorem \ref{thm:capacity_consumption} gives that every vertex solution fulfills the condition $x+d \in \Xi^{N \times N}$. 
\end{proof}

We can, however, make further statements regarding the vertices.

\begin{theorem}
    Let $d$ be feasible for the binary \eqref{eq:trip}  with $\delta,\beta,\gamma$ constructed as in Remark \ref{remark:sol construction}. Then $(d,\delta,\beta,\gamma)$ is a vertex of the polyhedron.
\end{theorem}
\begin{proof} {Proof}
    We construct our cost vector $c$ such that $(d,\delta,\beta,\gamma)$ is the only optimal point. We set
      \[
    c_{i,j} \coloneqq \Bigg\{\begin{matrix}
      -5 \alpha & \text{if $x_{i,j}+d_{i,j} = 1$}\\
       5 \alpha & \text{if $x_{i,j}+d_{i,j} = 0$}
    \end{matrix} 
  \] and it follows directly that every other choice of $d$ is suboptimal for the corresponding \eqref{eq:trip} problem. If a value $d_{i,j}$ is set to a different value than in the given point the first term in the objective would increase by $5 \alpha$ times the absolute value of this difference, because the $d_{i,j}$ value can only be changed in one direction. This is suboptimal because the second term can not be changed by more than $4 \alpha$ times the absolute value of this difference. 
\end{proof}
It is evident by the proof that this property extends to the case in which the $\beta, \gamma$, and $ \delta$ are weighted by multiplying the constructed $c$ with the maximum weight entry. 
However, this property does not translate to the case of a non-binary control even without weights as we show below.

\begin{example}
    Let $\Xi=\{0,1,2\}$, $N=M=3$, $\Delta = 3$, $\alpha=1$, and $x \equiv 0$. Define 
     \[d \coloneqq \begin{pmatrix}
        1 & 0 & 1 \\ 1 & 0 & 0 \\ 0 & 0 & 0
    \end{pmatrix}  \text{ and } d^1 \coloneqq \begin{pmatrix}
        0.5 & 0 & 2 \\ 0.5 & 0 & 0 \\ 0 & 0 & 0 
    \end{pmatrix} \text{and } d^2 \coloneqq \begin{pmatrix}
        1.5 & 0 & 0 \\ 1.5 & 0 & 0 \\ 0 & 0 & 0
    \end{pmatrix} 
    \text{ and }
    c\coloneqq \begin{pmatrix}
    -2 & 100 & -2.1 \\ -2 & 100 & 100 \\ 100 & 100 & 100
    \end{pmatrix}
    \] 
    We now show that $d$ with the corresponding $\beta,\gamma, \delta$ is not a vertex of the polyhedron $P$. We note that $d^1$ and $d^2$ are points in the polyhedron $P$ with the corresponding $\beta^1,\gamma^1, \delta^1$ and $\beta^2,\gamma^2, \delta^2$. It is obvious that $d= 0.5 d^1 +0.5 d^2$. Due to $x \equiv 0$ it also holds that $(\beta,\gamma,\delta) = 0.5 (\beta^1,\gamma^1,\delta^1) +   0.5 (\beta^2,\gamma^2,\delta^2) $. Thus $(d, \delta, \beta, \gamma)$ is not a vertex of the polyhedron $P$. Furthermore, $(d,\delta, \beta, \gamma)$ is optimal for the problem \eqref{eq:trip} with the cost vector c, which shows that the optimal solution to \eqref{eq:trip} does not have to be a vertex of the polyhedron. 
\end{example}

The relaxation $\eqref{eq:lr}$ provides a lower bound that is at least as good as the lower bound provided by
$\eqref{eq:lp}$. Both bounds are 
identical when the underlying polyhedron of the Lagrangian relaxation is already the convex hull of the integer-valued points, see e.g.\ p.\ 125 in \cite{korte2018}, which is the case for this problem as we have seen in the previous corollary. We now show that an optimal solution to \eqref{eq:lr} directly gives us an optimal solution to \eqref{eq:lp} and vice versa.

\begin{theorem}\label{thm:equivalenceLrLp}
    The problems \eqref{eq:lr} and \eqref{eq:lp} are equivalent in the sense that we can construct an optimal solution to one problem from 
    an optimal solution to the other problem.
\end{theorem}
\begin{proof} {Proof}
    Let $(d^*,\delta^*,\beta^*,\gamma^*)$ be optimal for \eqref{eq:lp}. We now want to construct a feasible point $(d,\delta,\beta,\gamma, \mu)$ for \eqref{eq:lr} and show the optimality afterwards. We start with the construction of $d$. We keep every entry in $d^*$ which is not fractional the same for the corresponding entry in $d$. If all entries are not fractional then $(d^*,\delta^*,\beta^*,\gamma^*,0)$ is already optimal for the integer program and thus also optimal for the Lagrangian relaxation which is a lower bound for the integer program and an upper bound for the linear programming relaxation. All fractional entries are either rounded up or rounded down to the same next value in $\Xi$ depending on which rounding step decreases the capacity consumption. The values of $\delta,\beta,\gamma$ are chosen as described in Remark \ref{remark:sol construction}. The value of $\mu$ is set as the quotient of the difference in the objective values regarding the cost function of \eqref{eq:ip} and the difference in the capacity consumption  of the two points $(d^*,\delta^*,\beta^*,\gamma^*)$ and $(d,\delta,\beta,\gamma)$. The objective value of $(d^*,\delta^*,\beta^*,\gamma^*)$ for the problem \eqref{eq:lp} and the objective value of  $(d,\delta,\beta,\gamma,\mu)$ regarding the objective function of \eqref{eq:lr} are identical which shows the optimality of $(d,\delta,\beta,\gamma,\mu)$.

    Now let $(d^1,\delta^1,\beta^1,\gamma^1, \mu)$ and $(d^2,\delta^2,\beta^2,\gamma^2,\mu)$ with $\sum_{i,j} \delta^1_{i,j} \leq \Delta$ and $\sum_{i,j} \delta^2_{i,j} \geq \Delta$ be optimal for \eqref{eq:lr}. In the case that both points are identical, the capacity constraint is fulfilled with equality and $(d^1,\delta^1,\beta^1,\gamma^1)$ is optimal for \eqref{eq:ip} and thus for \eqref{eq:lp}. In the case that both points are not identical the existence is ensured because otherwise we could improve the bound provided by increasing or decreasing $\mu$. We now use the convex combination $(d^3,\delta^3,\beta^3,\gamma^3,\mu)$ of the two points $(d^1,\delta^1,\beta^1,\gamma^1,\mu)$ and $(d^2,\delta^2,\beta^2,\gamma^2,\mu)$ which fulfills the capacity constraint with equality. The convex combination has the same objective value as the other two points regarding the objective function of \eqref{eq:lr}. Due to the capacity constraint being fulfilled with equality the point $(d^3,\delta^3,\beta^3,\gamma^3)$ also has the same objective value regarding the objective function of \eqref{eq:lp}. Thus this proves the optimality as the bound provided by the linear programming relaxation is not larger than the bound provided by the Lagrangian relaxation. 
\end{proof}

\begin{remark}
    In particular, the previous proof showed 
    for $\Xi = \{0,1\}$
    and a solution $d$ to the linear programming relaxation \eqref{eq:lp} that $d^t$ with 
    \begin{align*}
        x_{i,j} + d^t_{i,j} = \Bigg\{\begin{matrix}
            1 &\text{ if } x_{i,j} + d_{i,j} >t, \\
            0 &\text{ otherwise}
        \end{matrix}
    \end{align*}
    is optimal for every $t \in (0,1)$ for \eqref{eq:lr}. This result parallels  the
    infinite-dimensional thresholding result in \cite[Theorem 2.2]{burger2012exact}.
\end{remark}

\section{Integer programming solver-based solution}
\label{sec: IP based}
We have shown in Section \ref{subsec: NP hard} that the problem \eqref{eq:trip} is 
strongly NP-hard under the assumption that the minimum bisection problem is
NP-hard on grid graphs with an arbitrary number of holes. The latter has been conjectured in
\cite{Papadimitriou1990TheBW}. For the minimum bisection problem on solid subgraphs of the grid the best currently known algorithm has a run time of $\mathcal{O}(n^4)$, see \cite{Feldmann2011}. Even if the binary \eqref{eq:trip} is not NP-hard we believe it is likely that a polynomial algorithm would also have a high run time complexity. Due to these reasons we propose to employ an integer programming solver for solving \eqref{eq:trip}.
Based on our previous analysis, we derive several tools to reduce the run time of the integer programming solver. Specifically, we propose cutting planes, a primal heuristic, and a branching rule.

\subsection{Valid inequalities}
\label{subsec:cuts}
We introduce two different classes of cutting planes which use the structure of the polyhedron presented in Section \ref{sec: structure}, namely that we have one single fractional component. While the original constraints describing the polyhedron are very sparse except for the capacity constraint, the resulting cuts will not be sparse, but contain a number of variables depending on the size of the fractional component. As the latter may be as large as the whole graph the resulting inequalities would in turn be very dense. Thus, in computational practice, we have to add the cuts conservatively to ensure that the improvement of the linear relaxation is more impactful on the run time than the increase in computational demand for the relaxation.
While we will assume $\Xi = \{0,1\}$ for the derivation of the cutting planes, we believe that this may be relaxed to more general sets $\Xi$.

\subsubsection{Cutting plane derived from a fully connected graph}
\label{subsec:fully connected cut}
As shown in Section \ref{sec: structure} any vertex solution to \eqref{eq:lp}
has one maximal connected component with the same values $d+x$. Compared to any feasible $\Xi$-valued points with the same capacity consumption on this component the relaxation does not create any jumps while the $\Xi$-valued points do. To penalize this behavior
we can construct a cut which enforces that capacity consumption on the fractional component is reflected in the amount of jumps. The key ingredient is the minimum cut ratio on the fractional component. 

The notation in this as well as the following subsection is as follows. The tuple $(d^*,\delta^*,\beta^*,\gamma^*)$ denotes an arbitrary
but fixed optimal solution to \eqref{eq:lp}. The set of index pairs $F = \{(i,j) \in  \Omega \ | \ x_{i,j}+d_{i,j}^* \not \in \Xi \}$ contains those for which
the solution $x + d^*$ has fractional values. The set of index pairs $H = \{(i,j) \in \Omega \ | \ x_{i,j}+d_{i,j}^* \in \Xi, d_{i,j}^* \neq 0 \}$ contains
those for which the solution $x + d^*$ has values in $\Xi$ and a strictly positive capacity consumption (the corresponding entries of $d^*$ are non-zero).
For the capacity bound $\Delta$ and the capacity consumption outside of the fractional component
$\Delta_{out} \coloneqq \sum_{(i,j) \in \Omega \setminus F} \delta_{i,j}^* = \sum_{(i,j) \in H} \delta_{i,j}^*$, we obtain that
$\Delta_r \coloneqq \Delta - \Delta_{out}$ is the implied capacity bound for the fractional component. Due to 
Theorem \ref{thm:capacity_consumption}, this coincides with the capacity consumption of $d^*$ on $F$. We denote the largest connected component in $F$
in which the previous control values $x_{i,j}$ are identical by $G$.

The component $G$ can be interpreted as a connected subgraph of the grid. 
We define the cut ratio on this subgraph as 
\[ \rho \coloneqq \min_{U \subset G,\ 0<|U| \leq \Delta_r} \frac{|\partial U|}{|U|}\]
where $\partial U = \{ (v,w) \text{ is an edge in the subgraph} \ | \ v \in U, \ w \in G \setminus U \}$ is the set of cut edges from the graph partition into the sets $U$ and $G \setminus U$.  From the definition of $\rho$
it follows that $\rho |U| \le |\partial U|$
holds for all possible subsets $U\subset G$ of the fractional component
with capacity consumption $0 < |U| \le \Delta_r$. Consequently, every
feasible integer assignment on $G$ with capacity consumption less than
or equal to $\Delta_r$ corresponds to such a subset $U$.
The cut size $|\partial U|$ is the amount of jumps for this integer assignment
within $G$. Consequently, multiplying $\rho$ with the capacity consumption
of a feasible integer assignment that consumes less than or equal to
$\Delta_r$ capacity on $G$ gives a lower bound on its amount of jumps.

\begin{theorem}
	Let $\Xi = \{0,1\}$. Let $\Delta_{out}$ be a fixed integer value between $0$ and $\Delta$. Let $G, H$ be disjoint subgraphs of the grid such that $|H| = \Delta_{out}$ and $G$ is a connected component with the same value $x_{i,j}$ for every $(i,j) \in G$ and a minimum cut ratio $\rho$. Every feasible point $(d,\delta,\beta,\gamma)$ of \eqref{eq:ip} which fulfills
    $ \sum_{(i,j) \in H} \delta_{i,j} = \Delta_{out} $
    also fulfills the inequality
\begin{align*}
0 \leq - \rho \sum_{(i,j) \in G} \delta_{i,j} + \sum_{\substack{ ((i,j),(i+1,j)) \\ \in G \times G }} \beta_{i,j} + \sum_{ \substack{((i,j),(i,j+1)) \\ \in G \times G }} \gamma_{i,j}.
\end{align*}
For an $M$ sufficiently large the inequality
\begin{align*}
    0 \leq - \rho \sum_{(i,j) \in G} \delta_{i,j} + \sum_{\substack{((i,j),(i+1,j)) \\ \in G \times G} } \beta_{i,j} + \sum_{\substack{ ((i,j),(i,j+1)) \\ \in G \times G }} \gamma_{i,j} + M (\Delta_{out} - \sum_{(i,j) \in H} \delta_{i,j})
\end{align*}
holds for every feasible point of \eqref{eq:ip}.
\end{theorem}
\begin{proof}{Proof}
    The first part follows from the construction of $\rho$ and the insights presented above. The second part is the
    so-called big-M formulation of the implication. 
\end{proof}
We do not expect that we can calculate $\rho$ without significant computational demand for arbitrary subgraphs of the grid. Instead we want to use a simpler structure, a fully connected graph, for which we can determine the minimum cut ratio $\tilde \rho$ of the resulting subgraph in a straightforward manner. Thus we add the missing edges to the subgraph until we obtain the fully connected graph. For this graph we know that $|\partial U| = |G \setminus U| |U|$ and thus that the minimum cut ratio of the fully connected graph is $\tilde \rho = |G|-\Delta_r$. 

We now would have to add quadratically many variables to the constructed inequality in order to model all new edges which is not viable as it would significantly increase the solution time for the
linear programming relaxation. Instead we return to our original structure by replacing the new edges by weights on the original edges. 
Assume that the edge $e=(v,w)$ is added to construct the fully connected graph. If  $e$ is a cut edge then every path from $v$ to $w$ is also cut. We increase the weight by $1$ for all edges along a path from $v$ to $w$ in the original subgraph. We choose the path randomly among the set of shortest paths between the nodes. If we repeat this for every added edge then the sum of weighted jumps for the original edges bounds the amount of jumps for the edges of the fully connected graph from above. The success of the cutting plane also depends on $M$. We need to choose $M$ just large enough such that the inequality holds for feasible $\Xi$-valued points with a higher capacity use on the fractional component. This directly leads to the following theorem.

\begin{theorem}
 Let $\Xi = \{0,1\}$. 
 Let $\Delta_{out} \,\in \{0,\ldots,\Delta\}$. Let $G, H$ be disjoint subgraphs of the grid such that $|H| = \Delta_{out}$ and $G$ is a connected component with the same value $x_{i,j}$ for every $(i,j) \in G$ and $\tilde \rho = |G|-\Delta + \Delta_{out}>0$. Every feasible point $(d, \delta, \beta, \gamma)$ of \eqref{eq:ip} fulfills the inequality
    \begin{align}
    \label{ineq: fully connect}
    \tilde \rho \sum_{(i,j) \in G} \delta_{i,j} - M \Delta_{out} \leq &  \sum_{ \substack{((i,j),(i+1,j)) \\ \in G \times G }} w^\beta_{i,j} \beta_{i,j} + \sum_{ \substack{((i,j),(i,j+1)) \\ \in G \times G} }  w^\gamma_{i,j} \gamma_{i,j}  
     - M \sum_{(i,j) \in H} \delta_{i,j}
\end{align}
for all $M \ge M_0$, where
\begin{align*}
    M_0 \coloneqq &\tilde \rho - \min_{U \subset G, \Delta \geq |U|>\Delta_r} \frac{ (|G|-|U|) |U|  - (|G|-\Delta_r) \Delta_r }{|U|-\Delta_r} \\
    = & \tilde \rho -  \frac{ (|G|-\min\{ |G|,\Delta \} ) \min \{ |G|,\Delta \}  - (|G|-\Delta_r) \Delta_r }{ \min \{ |G|,\Delta \} -\Delta_r}.
\end{align*}
\end{theorem}
\begin{proof} {Proof}
    The inequality \eqref{ineq: fully connect} follows from the considerations above and using a big-M formulation. For the valid choice of $M$ we analyze both terms in its definition. The first term negates the additional effect of the first term on the left-hand side in the inequality \eqref{ineq: fully connect} if more than $\Delta_r$ capacity is used on the fractional component. The second term ensures that the inequality remains valid as the amount of jumps changes as more than $\Delta_r$ capacity is used on the fractional component. We can interpret the amount of jumps $|\partial U| = |G \setminus U| |U|$ as a concave function in $U$ meaning in the amount of capacity used on the fractional component and thus the construction of $M$ ensures that we affinely underestimate the amount of jumps for a given capacity.

\end{proof}

\begin{figure} 
    \hspace*{ -1 cm} 
    \centering
    \includegraphics[scale = 0.3]{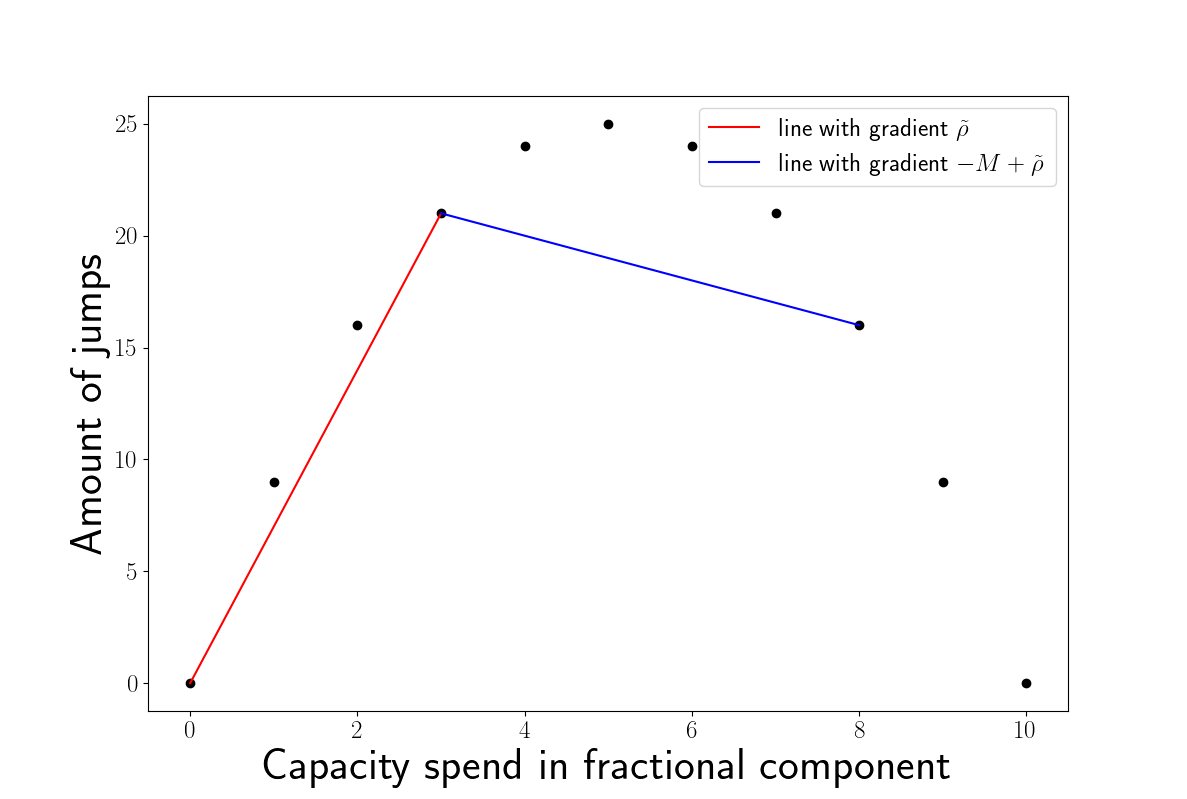}
    \caption{Visualization of the connection between the capacity used on the fractional component and the amount of jumps (cut edges) in the fully connected graph. For the example visualized above the values are given by $\Delta=8$, $\Delta_r = 3$ and $|G|=10$. We can see that $\tilde \rho$ and $M$ are chosen such that the amount of jumps is always affinely underestimated.}
    \label{fig:concave}
\end{figure}

For the optimal solution to \eqref{eq:lp}, the right hand side of the equation is equal to $-M \Delta_{out}$ as all $\beta_{i,j}$ and $\gamma_{i,j}$  are equal to zero and the sum of the $\delta_{i,j}$ regarding the subset $H$ is exactly $\Delta_{out}$ by construction. The left side of the equation is strictly larger than $M \Delta_{out}$ because $\tilde \rho >0$ and $\sum_{(i,j) \in G} \delta_{i,j} >0$. Thus this solution is cut off by the constructed inequality.

\subsubsection{Cutting plane derived from a bounding box}\label{sec:bounding_box_cut}
In the previous subsection we have used the fully connected graph to obtain a cut. It admits the drawback that the computational time to calculate the weights grows quadratically in the size of the fractional component. In the worst case the fractional component is identical to the whole graph. Thus
the trade off between the compute time of the cut and the run time reduction obtained from adding the cut might not be worth it. Instead we want to calculate a, potentially weaker, cut with significantly lower
compute time.

The idea is that instead of only considering the jumps in the fractional component we find the smallest bounding box containing the component and add a cut using the minimum cut ratio on this rectangular subgraph. We use the setting from the previous subsection and define $B \coloneqq \{ (i,j) \in \Omega \ | \ \exists (i,k), (\ell,j) \in F \}$ as the smallest bounding box containing $F$ for which we assume that the previous control in the nodes is $0$ for now. Furthermore, we define $H_B \coloneqq H \setminus B$ and $\Delta_B = \Delta- \sum_{(i,j) \in H_B} \delta_{i,j}^* $ and $\Delta_{out_B} = \sum_{(i,j) \in H_B} \delta_{i,j}^*$. 

We recall that for the bounding box we can underestimate the amount of cut edges for a given used capacity $K$ with the formula given in Lemma  \ref{lemma: box formula} which we already used for the NP-hardness conjectures. Note that in the original formulation our $B$ was called $G$.

\begin{theorem}
   Let $\Xi = \{0,1\}$.	
   Let $\Delta_{out_B}\ \in \{0,\ldots,\Delta\}$. Let $B$ be a $\tilde n \times \tilde m$ rectangular subgraph of the grid with $x_{i,j}=0$ for all $(i,j) \in B$ and $H_B$ be a subgraph of the grid with $|H_B|= \Delta_{out_B}$. Let $B$ and $H_B$ be disjoint. Then every feasible point of \eqref{eq:ip} fulfills
   \begin{align*}
       c_\rho \sum_{(i,j) \in B} \delta_{i,j} \leq \sum_{\substack{((i,j),(i+1,j)) \\ \in B \times B} } \beta_{i,j} + \sum_{\substack{ ((i,j),(i,j+1)) \\ \in B \times B }} \gamma_{i,j} + M (\Delta_{out_B}- \sum_{(i,j) \in H_B} \delta_{i,j}).
\end{align*}
with  $c_\rho \coloneqq \min_{0<K\leq \Delta-\Delta_{out_B}} \frac{\min \{ \lfloor \sqrt{K} \rceil + \lceil \sqrt{K} \rceil, \tilde n, \tilde m, \lfloor \sqrt{\tilde n \tilde m - K} \rceil +\lceil \sqrt{\tilde n \tilde m - K} \rceil \} }{K}$ and $M \coloneqq c_{\rho}+2$.
\end{theorem}
\begin{proof} {Proof}
    The Lemma \ref{lemma: box formula} showed that $c_\rho$ is a lower bound for the amount of jumps in a $\tilde n \times \tilde m$ rectangular graph if at most $\Delta-\Delta_{out_B}$ nodes are set to $1$. Thus, it follows that the inequality holds if the $M$ is chosen sufficiently large. In the proof of Lemma \ref{lemma: box formula} it was shown that the amount of jumps for $K$ nodes set to $1$ and $K+1$ nodes set to $1$ differs by at most $2$. It follows that the first term in the definition of $M$ negates the effect of the term on the left-hand side for values of $K$ larger than $\Delta-\Delta_{out_B}$ while the second term accounts for the highest possible decrease in the amount of jumps which is no more than $2$ for each additional node set to $1$.  
\end{proof}

The relevant computational demands are determining the bounding box and the capacity used by the solution to \eqref{eq:lp} in the nodes of the bounding box. These demands are linear in the size of the bounding box. In the worst case the size of the bounding box is quadratic in the size of the fractional component, but is bounded by the size of the grid which ensures that the calculation of the cut is significantly faster than the calculation of the previous cut for large fractional components.

\begin{remark}
    For the construction we have assumed that the previous controls in all nodes in the whole bounding box are $0$. This is however not needed and only done for the sake of clarity. If $c_2$ is the number of nodes with a previous control of $1$ the described cuts are valid if we instead use the formula $\min \{ \lfloor \sqrt{K} \rceil + \lceil \sqrt{K} \rceil, \tilde n, \tilde m, \lfloor \sqrt{\tilde n \tilde m - K-c_2} \rceil +\lceil \sqrt{\tilde n \tilde m - K-c_2} \rceil \}$ to underestimate the amount of jumps and only consider $\delta_{i,j}$ on the left-hand side with $x_{i,j}=0$. This is a direct consequence of applying the original formula in Lemma \ref{lemma: box formula} with $\tilde K = K + n_c$ for $n_c \in [0,c_2] \cap \mathbb{N}$.
\end{remark}

\subsection{Primal heuristics}
\label{subsec: primal}
We have already seen in Subsection \ref{subsec: connect graph} that substructures of the grid can correspond to one-dimensional problems of \eqref{eq:trip}. We want to use this property to improve any feasible point found by the integer programming solver. Let $(\tilde d, \tilde \delta, \tilde \beta, \tilde \gamma)$ be a feasible point of $\eqref{eq:ip}$ with $N=M$. We assume $N$ to be even for sake of clarity but the arguments hold for odd $N$. We consider the problem 
\begin{gather}\label{eq:ip-1D}
\begin{aligned}
    \min_{d,\delta,\beta,\gamma}\quad &  \sum_{i=1}^N \sum_{j=1}^M c_{i,j} d_{i,j} + \alpha  \left (\sum_{j=1}^{M} \sum_{i=1}^{N-1} \beta_{i,j} +  \sum_{i=1}^{N} \sum_{j=1}^{M-1} \gamma_{i,j} \right )\\
    \text{s.t.}\quad
    &  (d,\delta,\beta,\gamma) \in P_\Delta,\\    
    &  d \in \mathbb{Z}^{N \times M} \text{ and } d_{i,j} = \tilde d_{i,j} \text{ for } i \in \{2,4,\ldots,N\}, j \in \{1,\ldots, N\}.
\end{aligned}
\tag{red IP}
\end{gather}
Like the one-dimensional problem, the problem \eqref{eq:ip-1D} can be solved by a shortest-path approach with the same graph structure and a slight variation of the weights to consider the jumps to the fixed nodes, see Appendix \ref{appendix; primal improve} for more details. We now show that a solution of \eqref{eq:ip-1D} will always have an objective value no worse than the point $(\tilde d, \tilde \delta, \tilde \beta, \tilde \gamma)$  used for the construction of the problem.

\begin{theorem}
    Let $(\tilde d, \tilde \delta, \tilde \beta, \tilde \gamma)$ be a feasible point of $\eqref{eq:ip}$ with $N=M$. Then every solution of the problem \eqref{eq:ip-1D} is feasible for \eqref{eq:ip} and has a lower or equal objective value compared to the objective value of  $(\tilde d, \tilde \delta, \tilde \beta, \tilde \gamma)$ regarding the objective function of \eqref{eq:ip}.
\end{theorem}
\begin{proof} {Proof}
    The problem \eqref{eq:ip-1D} can be derived from \eqref{eq:ip} by adding the constraints
    $d_{i,j} = \tilde d_{i,j} \text{ for } i \in \{2,4,\ldots,N\}, j \in \{1,\ldots, N\}$.
    Thus the objective functions conincide and the feasible set of \eqref{eq:ip-1D} is a subset of \eqref{eq:ip}.
    The point $(\tilde d, \tilde \delta, \tilde \beta, \tilde \gamma)$ is feasible for \eqref{eq:ip-1D}. Thus any solution of \eqref{eq:ip-1D} has a objective value lower or equal to the objective value of $(\tilde d, \tilde \delta, \tilde \beta, \tilde \gamma)$ and is feasible for \eqref{eq:ip}.
\end{proof}

 We can now use this improved feasible point to construct a new problem of the form \eqref{eq:ip-1D} with different fixed $d$ entries and repeat this process. In our algorithm we alternately fix the entries $d_{i,j}$ with $i \in \{1,\ldots,N\}, j \in \{1,3, \ldots, N-1\}$, then $d_{i,j}$ with $i \in \{1,\ldots,N\}, j \in \{2,4, \ldots, N\}$, then $d_{i,j}$ with $i \in \{1,3,\ldots,N-1\}, j \in \{1, \ldots, N\}$ and finally  $d_{i,j}$ with $i \in \{2,4,\ldots,N\}, j \in \{1, \ldots, N\}$. We continue this process until none of the $4$ variations produce an improved feasible point. This process terminates finitely as there are only finitely many feasible points for the original problem $\eqref{eq:trip}$ and in each finite loop we improve the objective value.

\subsection{Branching rules}
\label{subsec: branch}
We propose a branching rule motivated by the primal heuristic.
In the previous subsection we observed that fixing half of the nodes allows us to obtain a one-dimensional problem and use a shortest path approach to obtain an optimal solution. This shows us that fixing the nodes in an order which follows this strategy might be preferable in order to come closer to a problem class that is more efficiently solvable as seen in Subsection \ref{subsec: primal}. Thus we choose to fix the nodes in even rows first as these nodes are most significant to progress towards the case that all even rows are fixed and we optimize over the remaining nodes.

\section{Computational experiments} \label{sec: comp exp}
To assess the performance of the tools analyzed in Section \ref{sec: IP based} we introduce an advection-diffusion problem in Subsection \ref{sec:benchmark} that serves as our benchmark problem. We run the SLIP algorithm proposed in \cite{leyffer2022sequential} for a uniform square grid of size $N \times N$ and binary controls $\Xi = \{0,1\}$ to produce subproblems of the form \eqref{eq:trip} for $5$ different values for the parameter $\alpha$. We compare $8$ different combinations of the proposed tools for the time to reach optimality and the gap closed after a given time. Further details are given in Subsection \ref{sec:setup}. Moreover, we examine if the size of the fractional component in the root linear programming relaxation is an indicator for the hardness of the problem \eqref{eq:trip} and approximation quality of the Lagrangian relaxation.

\subsection{Benchmark problem}\label{sec:benchmark}
Our advection-diffusion benchmark problem on $\Omega=(0,1)^2$ reads
\begin{gather}
\begin{aligned}
\min_{u,w} & \frac{1}{2} ||u-u_d||^2_{L^2(\Omega)} + \alpha \TV(w)\\
\text{s.t. } & - \varepsilon \Delta  u + c  \cdot  \nabla u = w \text{ in }\Omega,\ u |_{\Gamma_1} = 0,\ u|_{\Gamma_2}(x,y) = \sin(2\pi(x-0.25)),\\
& w(s) \in \{0,1\}
\end{aligned}    \label{diffusion} \tag{AD}
\end{gather}
with $\varepsilon=0.075$ and $c= (\cos(\pi / 32), \sin(\pi / 32)) )^T$. 
For the boundary we define
two subsets $\Gamma_1= ([0,0.25) \cup (0.75,1]) \times \{0\} \cup \{0,1\} \times (0,1) $  and $\Gamma_2 = [0.25,0.75] \times \{0\}$ for Dirichlet
boundary conditions. The remaining subset has a free boundary condition. We execute the SLIP algorithm
on discretizations of the domain and PDE. We use the python package FEniCSx, see \cite{BasixJossfenicsx,ScroggsEtal2022fenicsx,BarattaEtal2023fenicsx,AlnaesEtal2014fenicsx}, for the discretization of the domain, the PDE, and the gradient computation, where we follow a \emph{first-discretize, then-optimize} principle.

\subsection{Computational Setup}
\label{sec:setup}
We run the SLIP algorithm with the values $N=M \in \{32, 64\}$ and $\alpha \in \{4 \times 10^{-4}, 4\sqrt{5} \times 10^{-4}, 2\times 10^{-3}, 2\sqrt{5} \times 10^{-3}, 1 \times 10^{-2}  \}$ as well as $N=M=96$ and $\alpha \in \{ 4\sqrt{5} \times 10^{-4}, 2 \times 10^{-3}, 2\sqrt{5} \times 10^{-3} \}$, where we
chose only three values due to the computational demand. We set $\Delta^0= \frac{1}{16} N^2$ and choose $\sigma = 0.0001$ in Algorithm \ref{alg:slip}. To compare the approaches we solve each of the subproblems with the different combinations of tools that are detailed in Table \ref{tab: abbreviations}. We model the integer programs without the variables $\delta$ because they are not needed for the binary case as we can just use $d$ as the absolute value and add signs in the remaining inequalties depending on the previous control value. We employ the integer programming solver Gurobi 10.0.0, see \cite{gurobi}, and set a time limit of $1$ hour before returning the current best primal point in the subproblem solver to generate instances of the form \eqref{eq:trip}. We keep the default optimality tolerances of Gurobi which means a solution is considered optimal if the gap is less than 0.001. We include the time to build the model in Gurobi in our measurements. Thus the results include the whole run times of the subproblems but not the whole SLIP algorithm. The tools are implemented in C++ and we use pybind11, see \cite{pybind11}, to call the implemented functions from python. The cuts are added as lazy constraints to ensure that they are added to the model description. For a fair comparison the value \emph{PreCrush} is set to $1$ even when
no cuts are added because this setting significantly reduced the run time in our preliminary experiments. The laptop for the experiments has an Intel(R) Core i7(TM) CPU with eight cores
clocked at 2.5\,GHz and 64 GB RAM.
\begin{table}
\caption{Abbreviations for the different solution approaches (combinations of tools from Section \ref{sec: IP based}).}\label{tab: abbreviations}
\begin{center}
\begin{tabular}
{l | c c c c c c c c}
tools (Section) & none & p & b & c & p-b & p-c & b-c & p-b-c\\
\midrule
Primal (Section \ref{subsec: primal})  & - & x & - & - & x & x & - & x \\
Branching Rule (Section \ref{subsec: branch}) & - & - & x & - & x & - & x & x \\
Cuts (Section \ref{subsec:cuts}) & - & - &- & x & - & x & x &x 
\end{tabular}
\end{center}
\end{table}

\subsection{Comparison of the computational results for the different approaches}
 For the smallest value $N=32$ solving all $114$ subproblems, see Table \ref{tab: number inst}, only takes around $2$ to $3$ minutes. The approaches \emph{b-c} and 
 \emph{p-b-c} produce the best cumulative times with $115$ and $116$ seconds as seen in Table \ref{tab: total}. On average it takes both approaches 1 second to solve an instance of this size. In general, the approaches using cuts (\emph{c}, \emph{p-c}, \emph{b-c}, \emph{p-b-c}) perform better compared to the alternatives which is also reflected in the median run times depicted in Table \ref{tab: median} where the approach \emph{c} performs best. The approaches \emph{b-c} and \emph{p-b-c} perform best with a cumulative run time improvement of around 25 percent compared to the approach \emph{none}.

\begin{table} 
\centering
\caption{Number of instances produced by the SLIP algorithm for the different values for $N$ and $\alpha$.}
\label{tab: number inst}
\begin{tabular}
{r | c c c c c}
$N$ & $ \alpha =4  \times 10^{-4}$ & $\alpha = 4 \sqrt{5} \times 10^{-4} $ & $ \alpha = 2 \times 10^{-3}$ & $ \alpha = 2 \sqrt{5} \times 10^{-3}$ & $ \alpha  = 1 \times 10^{-2}$  \\
\midrule
$N=32$ & 40 & 34 & 24 &15 &1 \\
$N=64$ & 70 & 38 & 53 &19 & 1 \\
$N=96$ & - & 74 & 173 & 35 & - 
\end{tabular}
\end{table}

\begin{table} 
	\caption{Cumulative run times of the different approaches for $N \in \{32,64,96\}$ and the $5$ values for $\alpha$. We note that \emph{all} is the cumulative run time for all instances for the value of $N$.}
	\label{tab: total}
\centering
\begin{tabular}
{l r| r r r r r r r r}
N & $\alpha$ & none& p & b & c & p-b & p-c & b-c & p-b-c \\
\midrule
 32 & $4 \times 10^{-4} $ &\textbf{19} & 20 & 20 & 22 & 21 & 22 & 22 & 22 \\
 & $4 \sqrt{5} \times 10^{-4}$ & 23 & 23 & 23 & \textbf{22} & 24 & 23 & 23 & 23\\
 & $2 \times 10^{-3} $ &35 & 36 & 36 & \textbf{27} & 37 & \textbf{27} & \textbf{27} & \textbf{27}\\
 & $2 \sqrt{5} \times 10^{-3}$ &48 & 47 & 40 & 33 & 39 & 31 & \textbf{30} & \textbf{30} \\
 & $1 \times 10^{-2}$ &30 & 31 & 23 & 17 & 23 & 17 & \textbf{14} & \textbf{14}\\
  & all& 154 & 157 & 142 & 120 & 143 & 120 & \textbf{115} & 116  \\ 
\midrule
64 & $4 \times 10^{-4} $ &2905 & 2932 & 2646 & 1492 & 2670 & 1470 & 971 & \textbf{955} \\
 & $4 \sqrt{5} \times 10^{-4}$  &1540 & 1537 & 1296 & \textbf{594} & 1289 & 606 & 645 & 706\\
 & $2 \times 10^{-3} $ &7343 & 7437 & 5370 & 3625 & 5497 & 3531 & 2894 & \textbf{2876}\\
 & $2 \sqrt{5} \times 10^{-3}$ &16643 & 14554 & 8586 & 9326 & 8822 & 8874 & 6624 & \textbf{6448}\\
 & $1 \times 10^{-2}$ &\textbf{3600} & 3601 & 3601 & 3602 & 3601 & 3601 & 3601 & 3601\\
 &  all& 32032 & 30061 &21498 & 18639 & 21878 & 18082 & 14736 & \textbf{14586}  \\  
\midrule
  96& $4 \sqrt{5} \times 10^{-4}$ & 61292 & - & - & - & - & - & - & \textbf{46382}\\
 & $2 \times 10^{-3} $ &383256 & - & - & - & - & - & - & \textbf{350115}\\
 & $2 \sqrt{5} \times 10^{-3}$ &78743 & - & - & - & - & - & - & \textbf{65587} \\
 & all & 523292 & - & - & - & - & - & - & \textbf{462083}    
\end{tabular}
\end{table}

\begin{table}
\caption{Median run times of the different approaches for $N \in \{32,64,96\}$ and the $5$ values for $\alpha$. We note that \emph{all} is the median run time for all instances for the value of $N$.} 
\label{tab: median}	
\adjustbox{max width=\textwidth}{%
\centering
\begin{tabular}
{l r | r r r r r r r r}
N & $\alpha$ & none& p & b & c & p-b & p-c & b-c & p-b-c \\
\midrule
 32 & $4 \times 10^{-4} $ & \textbf{0.31} & 0.32 & 0.33 & 0.43 & 0.34 & 0.50 & 0.44 & 0.47 \\
 & $4 \sqrt{5} \times 10^{-4}$ &0.68 & 0.63 & 0.66 & \textbf{0.61} & \textbf{0.61} & 0.64 & 0.67 & 0.69\\
 & $2 \times 10^{-3} $ &1.28 & 1.24 & 1.29 & 1.00 & 1.32 & 1.00 & \textbf{0.89} & 0.90\\
 & $2 \sqrt{5} \times 10^{-3}$ &2.71 & 2.64 & 2.35 & 1.67 & 2.33 & 1.66 & \textbf{1.57} & 1.62 \\
 & $1 \times 10^{-2}$ &29.71 & 30.72 & 22.55 & 16.83 & 22.59 & 16.80 & 13.79 & \textbf{13.76}\\
 & all &0.74 & 0.77 & 0.78 & \textbf{0.67} & 0.83 & 0.69 & 0.71 & 0.72  \\ 
\midrule
 64 & $4 \times 10^{-4} $ &12.47 & 13.38 & 12.11 & 11.32 & 12.18 & 11.42 & \textbf{10.89} & 11.00 \\
 & $4 \sqrt{5} \times 10^{-4}$ &10.98 & 11.43 & 10.26 & 9.23 & 10.35 & 8.66 & 8.95 & \textbf{8.44}\\
 & $2 \times 10^{-3} $ &82.84 & 83.35 & 61.40 & 37.48 & 63.12 & 37.40 & 29.83 & \textbf{29.66}\\
 & $2 \sqrt{5} \times 10^{-3}$ &153.04 & 152.87 & 137.63 & 83.62 & 137.80 & \textbf{83.58} & 86.41 & 86.38 \\
 & $1 \times 10^{-2}$ &\textbf{3600.46} & 3600.57 & 3600.55 & 3602.37 & 3600.54 & 3601.04 & 3601.15 & 3601.15\\
 &  all & 26.78 & 26.88 & 23.00 & 16.95 & 23.77 & 16.33 & 15.02 & \textbf{14.40}  \\ 
\midrule
   96& $4 \sqrt{5} \times 10^{-4}$ & 298.61 & - & - & - & - & - & - & \textbf{121.39}\\
 & $2 \times 10^{-3} $ & 2617.74 & - & - & - & - & - & - & \textbf{2433.26}\\
 & $2 \sqrt{5} \times 10^{-3}$ &3600.93 & - & - & - & - & - & - & \textbf{1705.99}\\
 & all & 1579.09 & - & - & - & - & - & - & \textbf{989.70}  
\end{tabular}}
\end{table}
 This behaviour extends to the case with $N=64$ with an improvement of $46$ percent. For this discretization most instances are solved within $600$ seconds for the different approaches. We note that the approaches using the cuts (\emph{c}, \emph{p-c}, \emph{b-c}, \emph{p-b-c}) perform better than those which do not. In general, the approaches which include the primal heuristic (\emph{p}, \emph{p-b}, \emph{p-c}, \emph{p-b-c}) only seem to improve the run time if combined with the cuts (\emph{p-c}, \emph{p-b-c}). Both observations are illustrated by the performance plots in Figure \ref{fig:performance}. 
The best approach for both $N=32$ and $N=64$ is the combination of all the tools as the cumulative run times are lowest or second lowest, see Table \ref{tab: total}. For $N=64$ and $\alpha = 1 \times 10^{-2}$ the instance could not be solved by any time limit. All approaches produce the same primal point with objective value $0$. The objective lower bounds produced vary from $-0.18$ by \emph{none} to $-0.08$ by \emph{p-b-c}. 

\begin{figure}
\begin{subfigure}{0.5\textwidth}
\centering
\includegraphics[scale = 0.275]{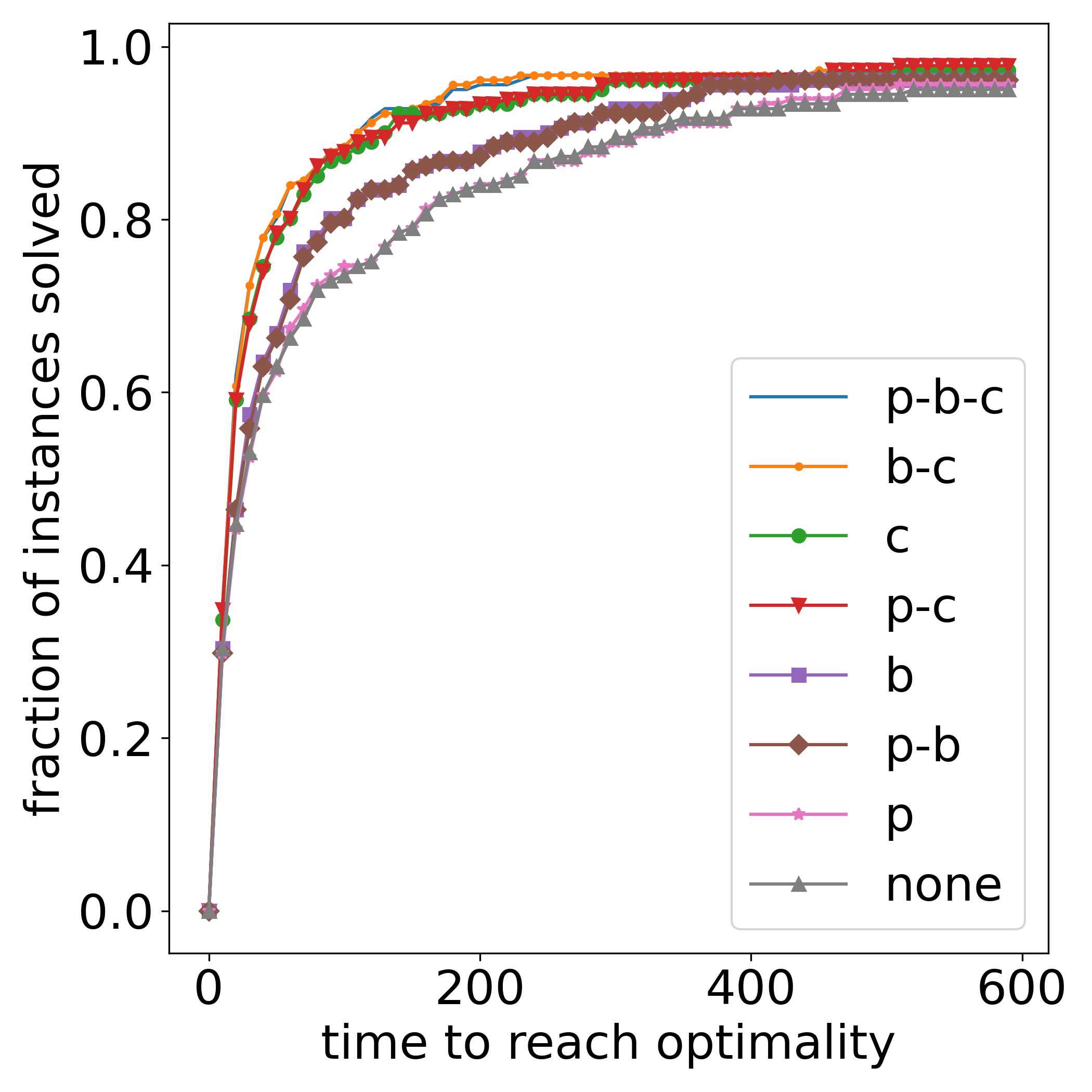}
\end{subfigure}
\begin{subfigure}{0.5\textwidth}
\centering
\includegraphics[scale = 0.278]{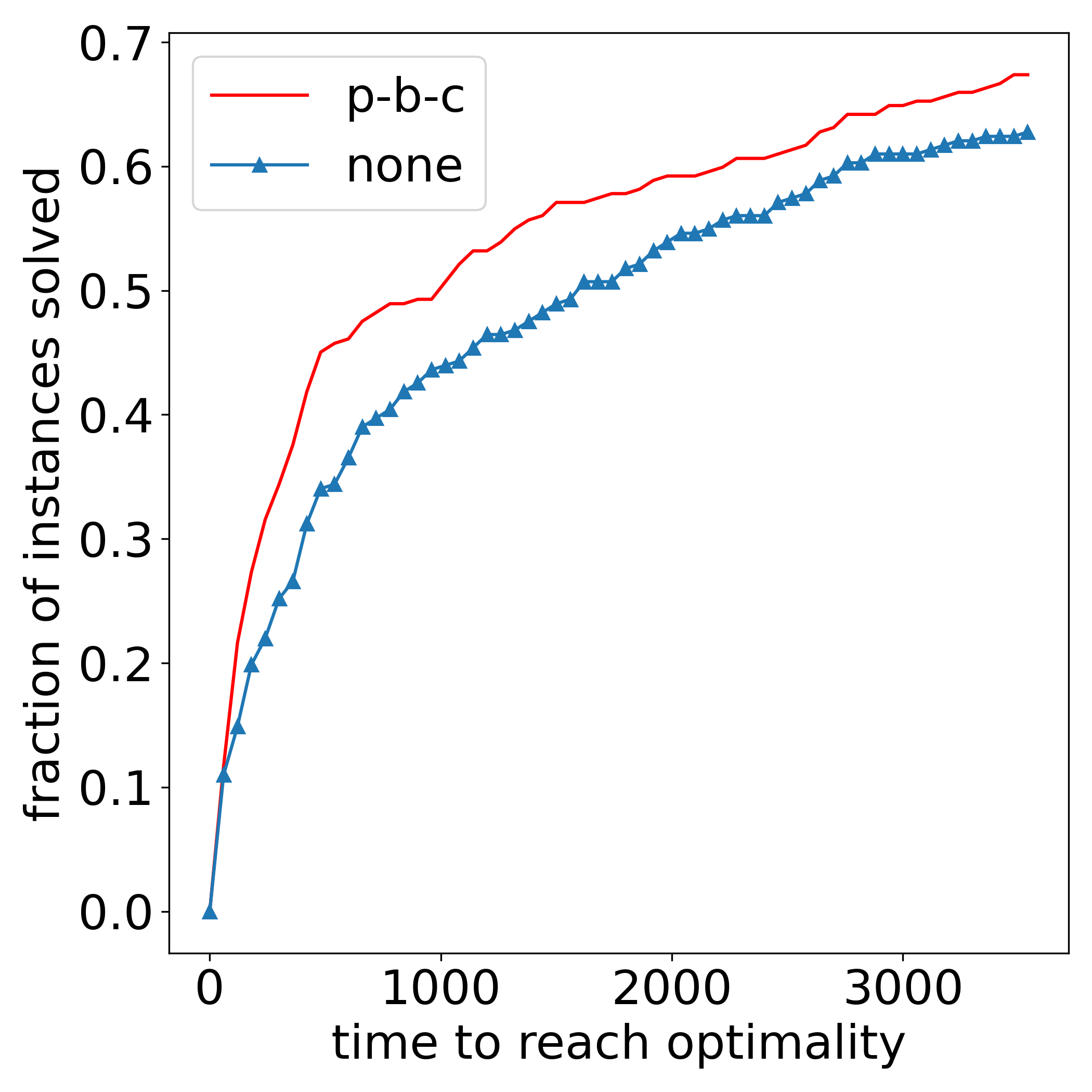}
\end{subfigure}
\caption{Performance plots for $N=64$ (left) and $N=96$ (right). The plots visualize what fraction of the instances are solved after a given time. For $N=64$ nearly all instances are solved after $600$ seconds while for $N=96$ a significant number of instances are not solved within the time limit of $1$ hour.}
\label{fig:performance}
\end{figure}

For $N=96$ we only compare the approaches with no tools and all tools combined. An improvement of $12$ percent for the cumulative run time is achieved by employing all proposed tools. We note that a larger improvement of $37$ percent is achieved regarding the median run times. This effect might be due to the significant number of instances reaching the time limit of $1$ hour which is visualized by the performance plot in Figure \ref{fig:performance}. 
There were $68$ instances which could not be solved within the time limit by either approach. Additionally $23$ instances could not be solved by \emph{p-b-c} but could be solved by \emph{none}, while the opposite case occurred for $37$ instances.
The mean duality gap is reduced from $2.6$ percent to $2.3$ percent by \emph{p-b-c}.  Furthermore, the highest occurring duality gap is also reduced as visualized by the violin plot in Figure \ref{fig:gap and violin}. The primal values were significantly better (exceeding the tolerance of the integer programming solver) for the approach "none" in $15$ cases while \emph{p-b-c} produced significantly better primal values in $23$ cases. 

We reran the experiments with a time limit of $3$ hours for all instances only solved by either \emph{p-b-c} or \emph{none} to get a clearer picture. We see that now the new cumulative run times of the approaches are 530781 seconds for \emph{p-b-c} and 655740 seconds for \emph{none}. Thus \emph{p-b-c} improves the run time by 19 percent compared to \emph{none}. We note that 3 instances for \emph{p-b-c} still reached the new time limit while 7 instances reached the new time limit for \emph{none}.

In many cases the feasible point obtained from the Lagrangian relaxation did not produce an approximation as no capacity was used. However, for $N=32$ there were $21$ instances for which at least half of the capacity was used by the feasible point. This effect gets smaller as $N$ increases as for $N=64$ there were $16$ such instances while for $N=96$ only $5$ instances produced such a feasible point from the Lagrangian relaxation. In these cases the approximation guarantee from  Theorem \ref{theorem: cond p} was always achieved. In $5$ cases the feasible point was optimal and in the remaining cases the feasible point was strictly better than the approximation guarantee derived in Theorem \ref{theorem: cond p}.

In Figure \ref{fig: size-time} we see that a larger fractional component in the root linear program corresponds to a higher run time. We note that for $N=96$ the linear regression is negatively impacted by the large amount of instances reaching the time limit and other regressions may be a better fit for the data but still shows the general trend.

\begin{figure}
\centering
\includegraphics[scale = 0.2]{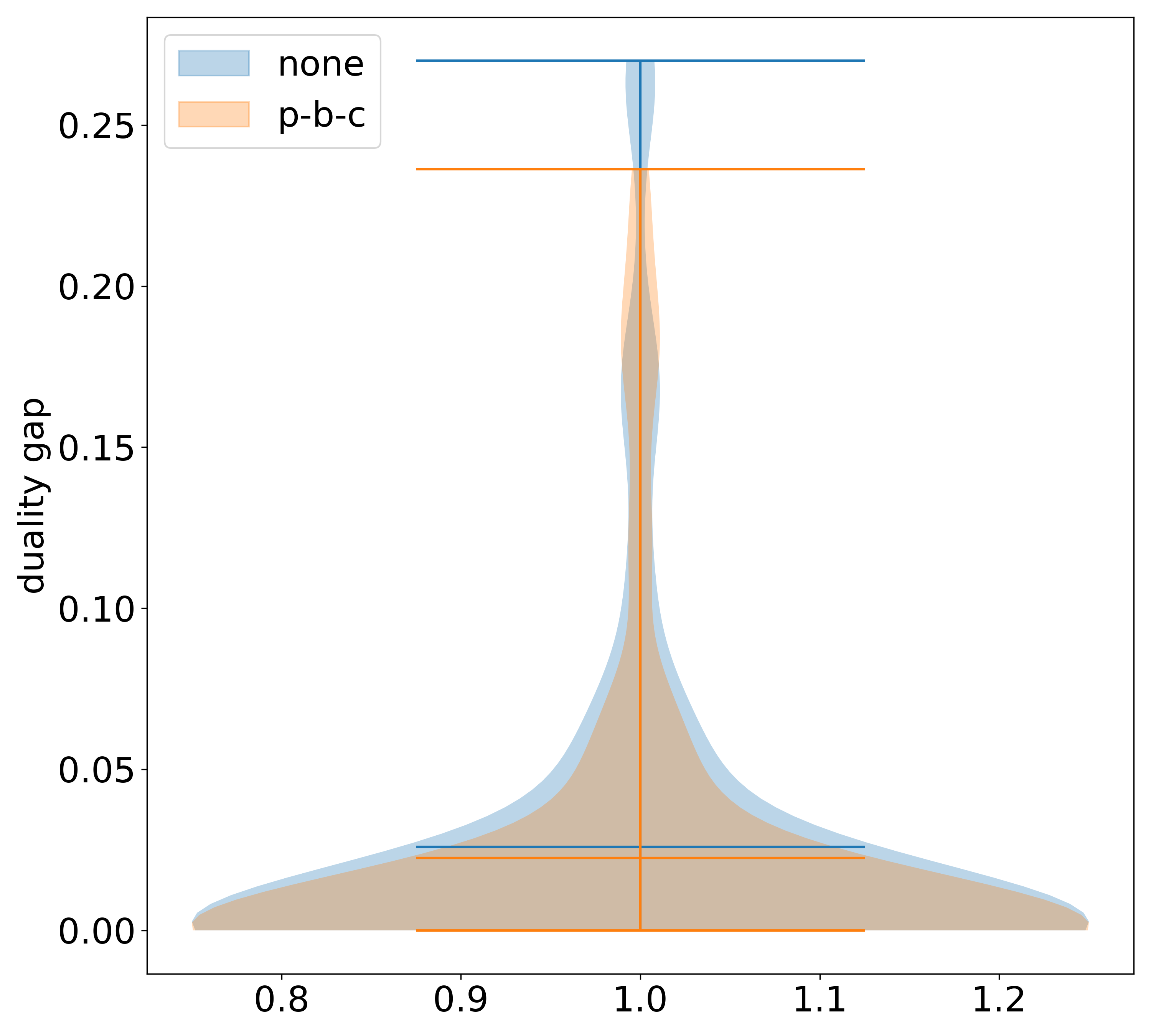} 
\caption{The violin plot shows the distribution of the remaining duality gap after the time limit of $1$ hour. The lines in the middle mark the mean duality gaps for all instances with $N=96$.}
\label{fig:gap and violin}
\end{figure}

\begin{figure}
\begin{subfigure}{0.32\textwidth}
\centering
\includegraphics[scale = 0.12]{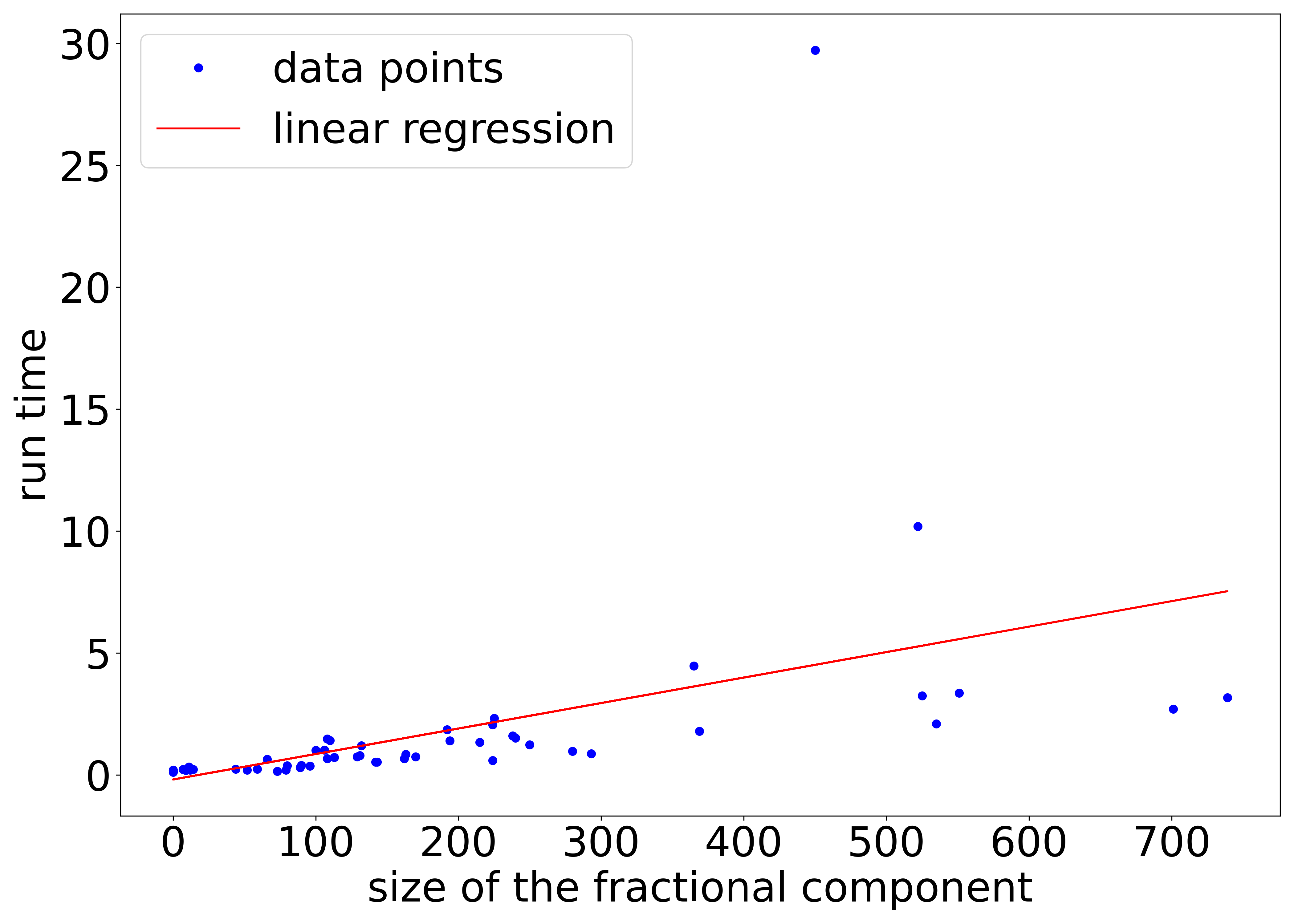}
\end{subfigure}
\hfill
\begin{subfigure}{0.32\textwidth}
\centering
\includegraphics[scale = 0.12]{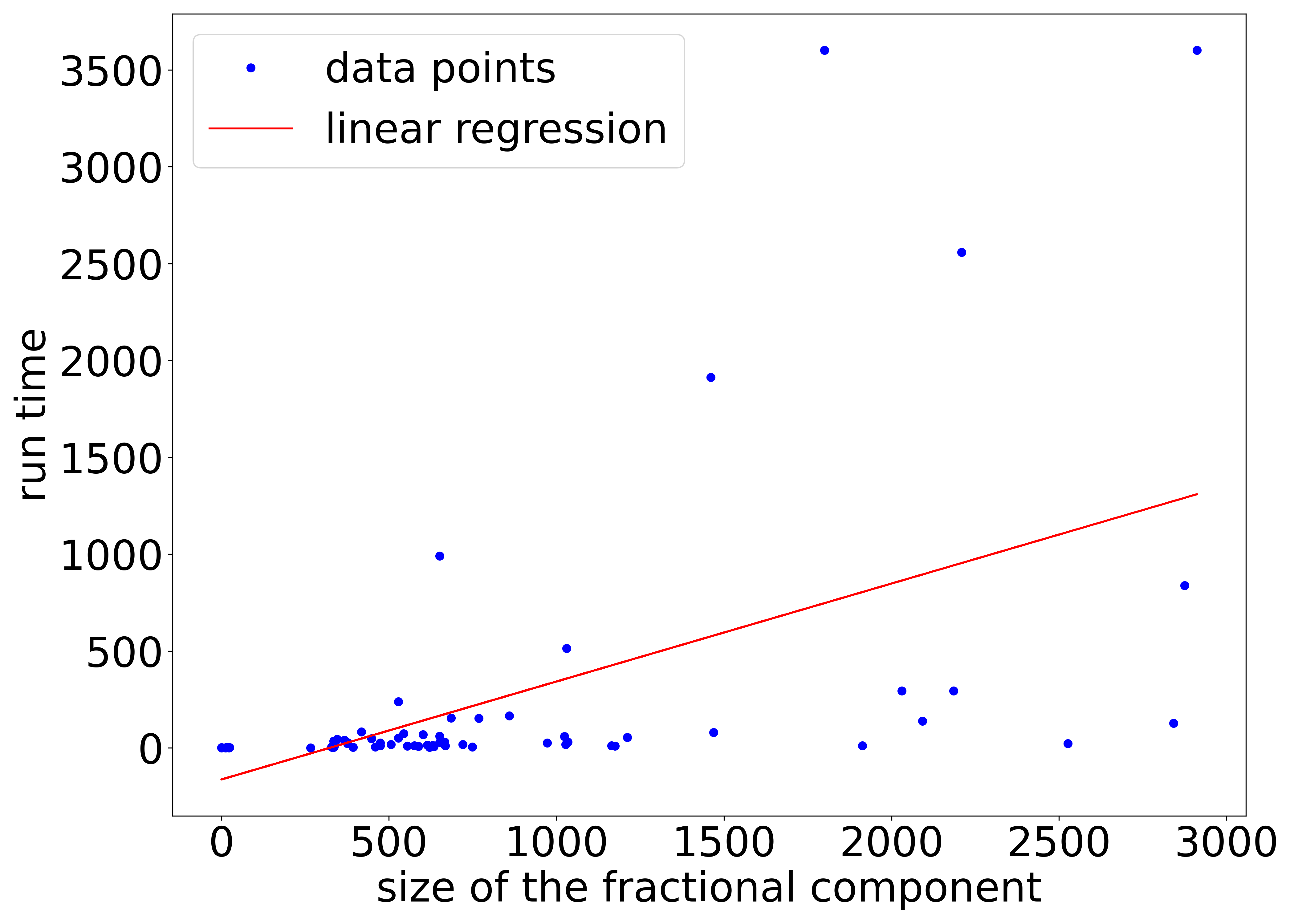}
\end{subfigure}
\hfill
\begin{subfigure}{0.32\textwidth}
\centering
\includegraphics[scale = 0.12]{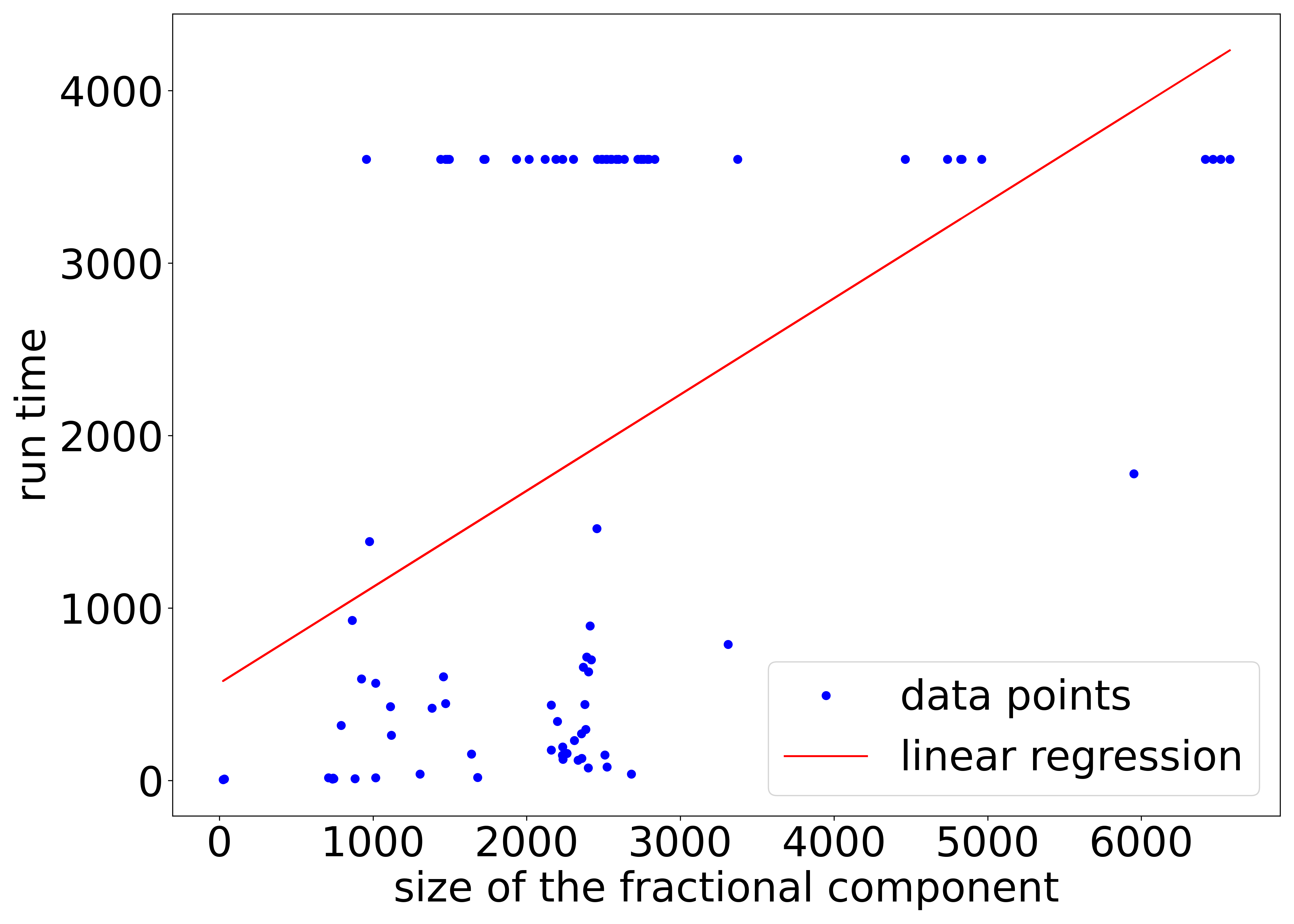}
\end{subfigure}
\caption{From left to right the plots show the relation between the size of the fractional component and the run time (none) for $N=32$ and $\Delta=64$, $N=64$ and $\Delta = 256$, $N=96$ and $\Delta=576$. In each plot the red line represents the linear regression of the data which itself is plotted as the blue points.}
\label{fig: size-time}
\end{figure}

\section{Conclusion}\label{sec: conclusion}
We have provided structural findings for the underlying polyhedron and its vertices as well as a conditional $p$-approximation.
Our findings have lead to tools to improve the run time of the integer programming solver. Our experimental results show that especially the proposed cutting planes reduce the run time substantially. Depending on the problem size the run time is reduced by up to $46$ \,\%. For larger values of $N$ this effect is reduced but improves if larger compute times are acceptable. We attribute this in parts to the fact that the inequalities describing the model only have a constant number of non-zero entries while the amount of non-zero entries in the added inequalities can grow with the size of the grid.

Our analysis and computational results motivate several avenues for future research.

The proposed cuts may be improved in two ways. First, for the cut using the bounding box
from Section \ref{sec:bounding_box_cut} we use the same coefficient for every node in the
fractional component in the added inequality. However, the construction of the cutting plane implies
that a minimum bisection which has to contain specific nodes on specific sides allows
for sharper coefficients and thus an improvement of the cut. Second, we have observed that the bounding
box can be significantly larger than the fractional component. In the worst case the size of the bounding box 
is quadratic in the size of the fractional component. Computing the actual minimum bisection of the 
fractional component instead of the bounding box would produce an inequality with fewer non-zero 
coefficients. 

In addition, cuts based on the fractional component itself also seem attractive because the component's size in the root linear program can indicate the hardness of the instance, see Figure \ref{fig: size-time}. 

While the Lagrangian relaxation provides a $p$-approximation, the feasible points were rarely useful 
as in many cases the used capacity is small and the approximation guarantee insufficient.

In order to leverage the achieved significant speed-up for medium-sized instances, we believe that
domain decomposition techniques on function space level are both attractive and viable so that one
obtains such instances of \eqref{eq:trip} in practice.

Moreover, as we have noted in Section \ref{sec: tr method}, the current discretization of the superordinate 
problem in function space currently implies an anisotropic discretization of the total variation.
Ongoing research shows that this may be overcome by successively adding linear inequalities to 
\eqref{eq:trip}.
Their effect on the problem structure and solution process is important for further research
and advancing the overall methodology but also significantly beyond the scope of this work.

We note that our results can still be applied when the problem contains additional fractional-valued control functions
if we impose the trust-region constraint $\max\{\|d_w\|_{L^1},\|d_u\|_{L^p}^p\} \le \Delta$, $p \ge 1$, where $d_w$ denotes
the step in the discrete-valued control and the $d_u$ denotes the step in the (newly added) fractional-valued control.
This is due to the fact that the trust-region subproblems completely decouple between these variables in this case if there
is no coupled regularization term. Coupling the discrete- and fractional-valued controls through the regularization term or
the trust-region may render the cutting planes invalid, however.

Regarding the dimension of the underlying domain (and the implied graph), we note that the suggested
branching rules and the primal heuristic can be transferred to higher dimensions straightforwardly but the implied
run time improvement is likely impaired because they rely on a serialization, which now contains less information
when the dimension increases. Moreover, the existence of the fractional component of the LP relaxation and the
cutting plane in Section \ref{subsec:fully connected cut} can still be shown with similar arguments
but the (computationally more efficient) bounding box-based cutting plane from Section \ref{sec:bounding_box_cut}
requires a multi-dimensional substitute for Lemma \ref{lemma: box formula}.

\section*{Acknowledgments}
The authors acknowledge funding by Deutsche Forschungsgemeinschaft (DFG) under grant no.\ MA 10080/2-1.

\appendix

\section{Proof of Lemma \ref{lemma: box formula}}\label{sec: proof box formula}
\begin{proof} {Proof}
    We start with the infinite grid $\mathbb{Z} \times \mathbb{Z}$ before returning to the rectangular subgraph. We want to determine a subset $U$ of the infinite grid with size $K=|U|$ such that the number of edges between $U$ and its complement in the infinite grid $U^C$ is minimized. 
    It is straightforward that $U$ is connected. If $U$ contains two separate connected components, we could shift one component until it becomes adjacent to the other component and reduce the amount of edges between $U$ and its complement in the infinite grid by at least one edge. Thus $U$ has to be a connected subset. 

    Let $v_\ell=(x_\ell,y_\ell), v_r=(x_r,y_r), v_b=(x_b,y_b), v_t=(x_t,y_t)$ be the leftmost, rightmost, bottommost, and topmost nodes in the set $U$. Then $U^C$ contains the node pairs $(x_\ell-1,y), (x_r+1,y)$ for $y \in [y_b,y_t]$. If we connect these pairs of the form $(x_\ell-1,y), (x_r+1,y)$ by a straight line we pass through at least one node in $U$ and thus obtain at least two cut edges per pair. The same holds true for the pairs $(x,y_b-1),(x,y_t+1)$ for $x \in [x_\ell,x_r]$. Thus there have to be at least $2(y_t-y_b)+2(x_r-x_\ell)$ cut edges between $U$ and $U^C$. For a set of size $K$ this implies that $2 \lceil \sqrt{K} \rceil + 2 \lfloor \sqrt{K} \rceil$ is a lower bound for the amount of edges between $U$ and $U^C$ (intuitively $U$ should as close to a square as possible).

    We can transfer these insights to the case of the positive orthant, the infinite grid $\mathbb{N} \times \mathbb{N}$. The argumentation from before also holds for the $\mathbb{N} \times \mathbb{N}$ grid but we now assume that for the leftmost node $v_\ell$ it holds that $v_\ell=(0,y_\ell)$ and for the bottommost node $v_b$ it holds that $v_b=(x_b,0)$ which implies that the minimum is halved meaning the minimum number of edges between $U$  and its complement in the $\mathbb{N} \times  \mathbb{N}$ grid is $y_t+x_r$ or $\lceil \sqrt{K} \rceil + \lfloor \sqrt{K} \rceil  $. 

    We now consider the case of the rectangular subgraph $G$. If $K< \min \{ \tilde n, \tilde m\} $ the situation coincides with the infinite grid $\mathbb{N} \times \mathbb{N}$. The concavity of the square root ensures that $U$ is connected because in the case of two components the minimum amount of edges would be twice the sum of the square roots of the number of nodes in the components. We obtain the same optimal structures and lower bounds for which examples are shown in the first two grids in Figure \ref{fig:optimal shapes}. If $K > |G|/2$ and $\tilde n \tilde m - K < \min \{ \tilde n, \tilde m \}$  we instead consider $U^C$ and obtain $ \lfloor \sqrt{\tilde n \tilde m - K} \rceil + \lceil \sqrt{\tilde n \tilde m - K} \rceil  $ edges. 
    
    It remains the case that $K \geq \min \{ \tilde n,\tilde m \}$ and  $\tilde n \tilde m -K \geq \min \{\tilde n,\tilde m \}$. Because edges to  $U^C$ may arise in only one of the four directions,
    we need to consider further optimal structures.
    If $K= c \tilde n$ for $c \in \mathbb{N}$, then we can choose $U$ as the first $c$ rows to obtain a set $U$ of size $K$ with exactly $\tilde n$ edges to its complement in $G$, see the third grid in Figure \ref{fig:optimal shapes}. The same argumentation holds if $K = c\tilde m$ for $c \in \mathbb{N}$. It is evident that this new structure is optimal iff $\tilde n <  \lceil \sqrt{K} \rceil +  \lfloor \sqrt{K} \rceil$ and $\tilde n < \lfloor \sqrt{\tilde n \tilde m - K} \rceil + \lceil \sqrt{\tilde n \tilde m - K} \rceil  $. If $(c+1) \tilde n >K>c \tilde n$ then either one of the structures derived from the infinite grid $\mathbb{N} \times \mathbb{N}$ is optimal or  we can choose $U$ as the first $c$ rows and the part of the next row such that the size matches which implies $\tilde n +1$ cut edges, see the fourth grid in Figure \ref{fig:optimal shapes}. Thus $\tilde n$ is also a lower bound in this case.
\end{proof}

\begin{figure}
    \centering
    
\resizebox{0.9 \textwidth}{!}{%
\begin{tikzpicture}[node distance={65mm}, thick, main/.style = {draw, circle}] 
 \foreach \x in {0,...,4} {
 \foreach \y in {0,...,4} {
    \node[main, minimum size = 0.5 cm] at(\x ,\y ) {};
}
}

\foreach \x in {0,...,4} {
\foreach \y in {0,...,3} {
\draw (\x, \y+0.25) -- (\x,\y+0.75);
}
}
\foreach \x in {0,...,3} {
\foreach \y in {0,...,4} {
\draw (\x+0.25, \y) -- (\x+0.75,\y);
}
}
\draw[line width=0.5mm,red](-0.25, 0.5) -- (1.5,0.5);
\draw[line width=0.5mm,red](1.5,0.5) -- (1.5,-0.25);
	\end{tikzpicture}
 \hspace{1cm}
\begin{tikzpicture}[node distance={65mm}, thick, main/.style = {draw, circle}] 
 \foreach \x in {0,...,4} {
 \foreach \y in {0,...,4} {
    \node[main, minimum size = 0.5 cm] at(\x ,\y ) {};
}
}

\foreach \x in {0,...,4} {
\foreach \y in {0,...,3} {
\draw (\x, \y+0.25) -- (\x,\y+0.75);
}
}
\foreach \x in {0,...,3} {
\foreach \y in {0,...,4} {
\draw (\x+0.25, \y) -- (\x+0.75,\y);
}
}
\draw[line width=0.5mm,red](-0.25, 1.5) -- (1.5,1.5);
\draw[line width=0.5mm,red](1.5,1.5) -- (1.5,-0.25);
	\end{tikzpicture}
  \hspace{1cm}
 \begin{tikzpicture}[node distance={65mm}, thick, main/.style = {draw, circle}] 
 \foreach \x in {0,...,4} {
 \foreach \y in {0,...,4} {
    \node[main, minimum size = 0.5 cm] at(\x ,\y ) {};
}
}

\foreach \x in {0,...,4} {
\foreach \y in {0,...,3} {
\draw (\x, \y+0.25) -- (\x,\y+0.75);
}
}
\foreach \x in {0,...,3} {
\foreach \y in {0,...,4} {
\draw (\x+0.25, \y) -- (\x+0.75,\y);
}
}
\draw[line width=0.5mm,red](-0.25, 1.5) -- (4.25,1.5);
	\end{tikzpicture}
  \hspace{1cm}
  \begin{tikzpicture}[node distance={65mm}, thick, main/.style = {draw, circle}] 
 \foreach \x in {0,...,4} {
 \foreach \y in {0,...,4} {
    \node[main, minimum size = 0.5 cm] at(\x ,\y ) {};
}
}

\foreach \x in {0,...,4} {
\foreach \y in {0,...,3} {
\draw (\x, \y+0.25) -- (\x,\y+0.75);
}
}
\foreach \x in {0,...,3} {
\foreach \y in {0,...,4} {
\draw (\x+0.25, \y) -- (\x+0.75,\y);
}
}
\draw[line width=0.5mm,red](-0.25, 1.5) -- (2.5,1.5);
\draw[line width=0.5mm,red](2.5, 1.5) -- (2.5,2.5);
\draw[line width=0.5mm,red](2.5, 2.5) -- (4.5,2.5);
	\end{tikzpicture}
}

\caption{Visualization of the optimal cut for different sizes of the set $U$. From left to right: $K=2$, $K=4$, $K=10$ and $K=12$. For the first two values of $K$ a shape as close to a square is desirable while for $K=10$ as an integer multiple of the row size of $5$ choosing the first two bottom rows as the set $U$ is optimal. For $K=12$ we set the first two rows to $1$ as well as $2$ nodes in the third row.}
\label{fig:optimal shapes}
\end{figure}
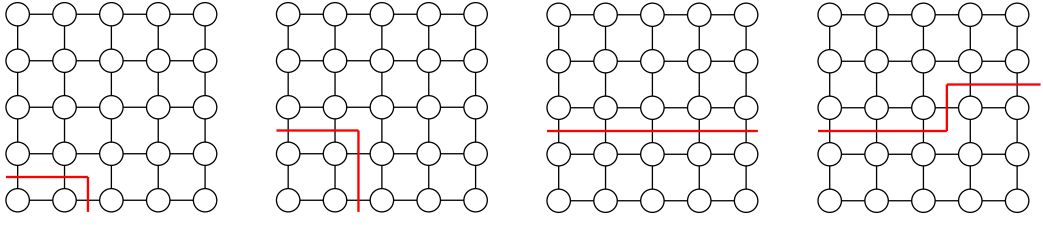

\section{Dual decomposition relaxation}
\label{subsec: dual decomp relax}
We introduce another relaxation which did not prove useful in our preliminary experiments but which might prove useful with slight variations in the future.
In \eqref{eq:ip} the terms modeling the total variation are split into sums over the auxiliary variables $\beta \in \mathbb{N}^{(N-1)\times M}$ and $\gamma \in \mathbb{N}^{N\times (M-1)}$. The corresponding linear inequalities are coupled through $d$. To obtain a new relaxation we rewrite \eqref{eq:ip} such that we split each entry in $d$ into two copies $d^c$ and $d^r$ where $d^c$ is used for the inequalities containing $\beta$ and $d^r$ for the inequalities containing $\gamma$. This will allow us to solve the resulting problems in polynomial time. Furthermore, this will provide a potentially tighter lower bound than the the one obtained from \eqref{eq:lp}.
For the underlying grid illustrated in Figure \ref{fig:grid} this means that $d^r$ considers only the rows while $d^c$ considers only the columns. To ensure equivalence to \eqref{eq:ip} we add coupling equalities to enforce that both copies are equal. We obtain:
\begin{gather}\label{eq:ip_dd}
\begin{aligned}
    \min_{d^r,d^c,\delta^r, \delta^c,\beta,\gamma}\quad &  \frac{1}{2} \sum_{i=1}^N \sum_{j=1}^M c_{i,j} d_{i,j}^r+ \frac{1}{2} \sum_{i=1}^N \sum_{j=1}^M c_{i,j} d_{i,j}^c + \alpha  \left (\sum_{j=1}^{M} \sum_{i=1}^{N-1} \beta_{i,j} +  \sum_{i=1}^{N} \sum_{j=1}^{M-1} \gamma_{i,j} \right )\\
    \text{s.t.}\quad
    &  (d^r,\delta^r,\gamma) \in P_r,\ (d^c,\delta^c,\beta) \in P_c,\ d^r,\ d^c \in \mathbb{Z}^{N \times M},
    \text{ and }d^c = d^r,
\end{aligned}
\tag{IP-DD}
\end{gather}
where $P_r$ is the polyhedron defined by 
\begin{align*}
     (d^r, \delta^r,\gamma) \in P_r  &\logeq  \begin{cases}
         \min \Xi \leq x_{i,j} + d^r_{i,j} \leq \max \Xi &\text{ for all } i \in [N], j \in [M],\\
     -\gamma_{i,j} \leq x_{i,j+1} + d^r_{i,j+1} - x_{i,j} - d^r_{i,j} \leq \gamma_{i,j} &\text{ for all } i \in [N], j \in [M-1],\\
    -\delta^r_{i,j} \leq d^r_{i,j} \leq \delta^r_{i,j} &\text{ for all } i \in [N], j \in [M],\\
     \sum_{i=1}^N \sum_{j=1}^M  \delta^r_{i,j} \le \Delta,
    \end{cases} 
\end{align*}
and  $P_c$ is the polyhedron defined by 
\begin{align*}
    (d^c, \delta^c,\beta) \in P_c  &\logeq  \begin{cases}
         \min \Xi \leq x_{i,j} + d^c_{i,j} \leq \max \Xi &\text{ for all } i \in [N], j \in [M],\\
     -\beta_{i,j} \leq x_{i+1,j} + d^c_{i+1,j} - x_{i,j} - d^c_{i,j} \leq \beta_{i,j} &\text{ for all } i \in [N-1], j \in [M],\\
     -\delta^c_{i,j} \leq d^c_{i,j} \leq \delta^c_{i,j} &\text{ for all } i \in [N], j \in [M],\\
     \sum_{i=1}^N \sum_{j=1}^M  \delta^c_{i,j} \le \Delta.
    \end{cases} 
\end{align*}
We briefly state that the constructed problem \eqref{eq:ip_dd} is equivalent to \eqref{eq:ip}.
\begin{lemma}\label{lem:ipdd_is_ip}
    Let $(d,\delta,\beta,\gamma)$ be feasible for \eqref{eq:ip}. Then $(d^r,d^c,\delta^r,\delta^c,\beta,\gamma)$ with $d^r=d^c=d$, $\delta^r=\delta^c=\delta$ is feasible for \eqref{eq:ip_dd} with the same objective value.
    If $(d^r,d^c,\delta^r,\delta^c,\beta,\gamma)$ is feasible for \eqref{eq:ip_dd} then $(d,\delta,\beta,\gamma)$ with $d=d^r$ and $\delta = \delta^r$ is feasible for \eqref{eq:ip} with the same objective value.
\end{lemma}
\begin{proof} {Proof}
    Let $(d,\delta,\beta,\gamma)$ be feasible for \eqref{eq:ip}. We show that $(d^r,d^c,\delta^r,\delta^c,\beta,\gamma)$ with $d^r=d^c=d$ and $\delta^r=\delta^c=\delta$ is feasible for \eqref{eq:ip_dd}. We first note that $ (d^r, \delta^r,\gamma) \in P_r$, which follows from $d=d^r$, $\delta = \delta^r$, and $(d,\delta,\beta,\gamma) \in P_\Delta$.
 By the same argumentation $ (d^c, \delta^c,\beta) \in P_c$. The remaining constraints are obviously fulfilled. The objective values coincide because $d^r=d^c=d$ so that
\begin{align*}
    \frac{1}{2} \sum_{i=1}^N \sum_{j=1}^M c_{i,j} d_{i,j}^r+ \frac{1}{2} \sum_{i=1}^N \sum_{j=1}^M c_{i,j} d_{i,j}^c 
     =  \sum_{i=1}^N \sum_{j=1}^M c_{i,j} d_{i,j}
\end{align*}  
and the remaining terms are the same.
The other implication can be proven in the same way by just switching the roles of $d^r$ and $d$ as well as of $\delta^r$ and $\delta$ and using $d^r=d^c$.
\end{proof}

\begin{corollary}
\label{cor: lpeq}
    Let $(d,\delta,\beta,\gamma)$ be a feasible point of \eqref{eq:lp}. Then $(d^r,d^c,\delta^r,\delta^c,\beta,\gamma)$ with $d^r=d^c=d$ and $\delta^r=\delta^c=\delta$ is feasible for the linear programming relaxation of \eqref{eq:ip_dd} with the same objective value.
    On the other hand if $(d^r,d^c,\delta^r,\delta^c,\beta,\gamma)$ is feasible for the linear programming relaxation of \eqref{eq:ip_dd} then $(d,\delta,\beta,\gamma)$ with $d=d^r$ and $\delta = \delta^r$ is feasible for \eqref{eq:lp} with the same objective value. 
\end{corollary}
\begin{proof} {Proof}
    We argue as in Lemma \ref{lem:ipdd_is_ip}. We just need to drop the integrality conditions from both problems.
\end{proof}

\begin{remark}
    The previous lemma and corollary also hold for $d=d^c$ and $\delta = \delta^c$ by symmetry of the problem.
\end{remark}
We can now obtain a Lagrangian relaxation by moving the coupling constraint $d^r=d^c$ into the objective with a multiplier variable. The problem reads
\begin{gather}\label{eq:lr_dd}
\begin{aligned}
    \min_{d^r,d^c,\delta^r, \delta^c,\beta,\gamma}\quad &  \frac{1}{2} \sum_{i=1}^N \sum_{j=1}^M c_{i,j} d_{i,j}^r+ \frac{1}{2} \sum_{i=1}^N \sum_{j=1}^M c_{i,j} d_{i,j}^c  + \alpha  \left (\sum_{j=1}^{M} \sum_{i=1}^{N-1} \beta_{i,j} +  \sum_{i=1}^{N} \sum_{j=1}^{M-1} \gamma_{i,j} \right )\\& + \sum_{i=1}^N \sum_{j=1}^M \lambda_{i,j} (d^r_{i,j}-d^c_{i,j})\\
    \text{s.t.}\quad
    &  (d^r,\delta^r,\gamma) \in P_r,\  (d^c,\delta^c,\beta) \in P_c,\  d^r,\ d^c \in \mathbb{Z}^{N \times M}
\end{aligned}
\tag{LR-DD}
\end{gather}
and provides a lower bound on the objective of \eqref{eq:ip_dd} for every $\lambda \in \mathbb{R}^{N \times M}$. 

\begin{theorem}
    Let $\lambda \in \mathbb{R}^{N \times M}$ be fixed. Then the optimal objective value of \eqref{eq:lr_dd} provides a lower bound for the optimal objective value of \eqref{eq:ip_dd}. 
\end{theorem}
\begin{proof} {Proof}
    Let $(d^r,d^c,\delta^r,\delta^c,\beta,\gamma)$ with $d^r=d^c$ be an arbitrary, feasible point of \eqref{eq:ip_dd}. Then this point is also feasible for \eqref{eq:lr_dd}. Because $d^r=d^c$ the equality $\sum_{i=1}^N \sum_{j=1}^M \lambda_{i,j} (d^r_{i,j}-d^c_{i,j})=0$ holds and the objective function values of \eqref{eq:ip_dd} and \eqref{eq:lr_dd} coincide. Thus the optimal objective value of \eqref{eq:lr_dd} can not be higher than the optimal objective value of \eqref{eq:ip_dd}.
\end{proof}

The problem \eqref{eq:lr_dd} can be decoupled into the linear programs
\begin{gather}\label{r_dd}
\begin{aligned}
    \min_{d^r,\delta^r,\gamma}\quad &  \frac{1}{2} \sum_{i=1}^N \sum_{j=1}^M (c_{i,j}+ 2\lambda_{i,j})  d_{i,j}^r + \alpha  \sum_{i=1}^{N} \sum_{j=1}^{M-1} \gamma_{i,j}  \\
    \text{s.t.}\quad
    &  (d^r,\delta^r,\gamma) \in P_r \text{ and } d^r \in \mathbb{Z}^{N \times M},
\end{aligned}
\tag{R-DD}
\end{gather}
and 
\begin{gather}\label{c_dd}
\begin{aligned}
    \min_{d^c,\delta^c,\beta}\quad &  \frac{1}{2} \sum_{i=1}^N \sum_{j=1}^M (c_{i,j}-2\lambda_{i,j}) d_{i,j}^c + \alpha  \sum_{i=1}^{N-1} \sum_{j=1}^{M} \beta_{i,j}  \\
    \text{s.t.}\quad
    &  (d^c,\delta^c,\beta) \in P_c \text{ and }  d^c \in \mathbb{Z}^{N \times M},
\end{aligned}
\tag{C-DD}
\end{gather}
which can be solved independently for $d^r,\delta^r, \gamma$ and $d^c,\delta^c,\beta$ in order to solve \eqref{eq:lr_dd}.
Each of the resulting problems is a one-dimensional version of \eqref{eq:trip}. We have already shown that the sum of the optimal values provides a lower bound. We are now interested in cases, where this allows to construct an optimal point for \eqref{eq:ip_dd} and hence \eqref{eq:ip}.

\begin{theorem}
    Let $(d^{r,*}, \delta^{r,*}, \gamma^*)$ be optimal for the problem \eqref{r_dd} and $(d^{c,*}, \delta^{c,*}, \beta^*)$ be optimal for \eqref{c_dd}. If $ d^{r,*} = d^{c,*}$ holds, then $(d^*,\delta^*,\beta^*,\gamma^*)$ with $d^*= d^{r,*}$ and $\delta^* = \delta^{r,*}$ is optimal for \eqref{eq:ip_dd} and hence \eqref{eq:ip}.
\end{theorem}
\begin{proof} {Proof}
    The feasibility of the constructed point follows by construction because $d^{r,*}$ and $d^{c,*}$ adhere to the capacity constraint and the controls can only take control values in $\Xi$. Thus it remains to show the optimality.
    
    Let $(d,\delta,\beta,\gamma)$ be feasible for \eqref{eq:ip}. Let $(d^{r,*}, \delta^{r,*}, \gamma^*)$ be an optimal point for \eqref{r_dd} and $(d^{c,*}, \delta^{c,*}, \beta^*)$ be an optimal point for \eqref{c_dd}. Then the two inequalities 
    \[  \sum_{i=1}^N \sum_{j=1}^M (\frac{1}{2} c_{i,j}-\lambda_{i,j}) d_{i,j} + \alpha  \sum_{i=1}^{N} \sum_{j=1}^{M-1} \gamma_{i,j}   \geq  \sum_{i=1}^N \sum_{j=1}^M (\frac{1}{2} c_{i,j}-\lambda_{i,j}) d_{i,j}^{r,*} + \alpha  \sum_{i=1}^{N} \sum_{j=1}^{M-1} \gamma_{i,j}^*    \]
    and 
    \[  \sum_{i=1}^N \sum_{j=1}^M (\frac{1}{2} c_{i,j}+\lambda_{i,j}) d_{i,j} + \alpha  \sum_{i=1}^{N-1} \sum_{j=1}^{M} \beta_{i,j}     \geq   \sum_{i=1}^N \sum_{j=1}^M (\frac{1}{2} c_{i,j}+\lambda_{i,j}) d_{i,j}^{c,*} + \alpha  \sum_{i=1}^{N-1} \sum_{j=1}^{M} \beta_{i,j}^*    \]
  hold. The statement follows from adding both inequalities.
\end{proof}
\begin{remark}
    The most straightforward split is into one row and one column problem. Other splits of $\beta$ and $\gamma$ into two sets are also possible but might create subproblems which are not known to be pseudo-polynomially solvable.
\end{remark}

The bound provided by a Lagrangian relaxation with an optimal Lagrange multiplier is at least as good as the bound provided by the linear programming relaxation, see \cite{geoffrion1974}. Thus, Corollary \ref{cor: lpeq} gives that the bound from \eqref{eq:lr_dd} for an optimal $\lambda$, which maximizes the objective of \eqref{eq:lr_dd}, is at least as good as the bound from \eqref{eq:lp}. We provide a minimal example to show that the dual decomposition relaxation can also provide a tighter bound.
\begin{example}
     Let $N=M=2$, $\Delta=1$, $\alpha=1$, \[c=  \begin{pmatrix}
        -0.5  & -0.5 \\ -0.5 & -0.5
    \end{pmatrix}, \ x= \begin{pmatrix}
        0  & 0 \\ 0 & 0
    \end{pmatrix}.\] Then the unique optimal point of \eqref{r_dd} for $\lambda \equiv 0$ is $d^r \equiv 0 $ and of \eqref{c_dd} for  $\lambda \equiv 0 $ is $d^c \equiv 0$. Thus $d \equiv 0$ is optimal for \eqref{eq:ip} but the linear programming relaxation solution is \[d=\begin{pmatrix}
        0.25  & 0.25 \\ 0.25 & 0.25
    \end{pmatrix}\] which shows that the dual decomposition relaxation is tighter in this case.
\end{example}

We believe that the relaxation did not produce better bounds in preliminary experiments because the capacity was significantly larger than each row or column. Thus it is possible to model the fractional component by setting the corresponding entries in some rows/columns to $1$ and the entries in the remaining rows/columns to $0$. The example above is chosen in such a way that this can not occur. We hypothesise that a different split which does not have this weakness may produce better bounds at the cost of an significantly increased computational demand.
\section{Graph construction for the primal improvement algorithm}
\label{appendix; primal improve}
In this subsection we show how the problems arising in the primal improvement approach can be solved as a shortest path problem.
Let $(\tilde d, \tilde \beta, \tilde \gamma, \tilde \delta)$ be a feasible point of $\eqref{eq:ip}$ with $N=M$. We assume $N$ to be even for sake of clarity but the arguments also hold for odd $N$. We now consider the problem 
\begin{gather}
\begin{aligned}
    \min_{d,\delta,\beta,\gamma}\quad &  \sum_{i=1}^N \sum_{j=1}^N c_{i,j} d_{i,j} + \alpha  \left (\sum_{j=1}^{N} \sum_{i=1}^{N-1} \beta_{i,j} +  \sum_{i=1}^{N} \sum_{j=1}^{N-1} \gamma_{i,j} \right )\\
    \text{s.t.}\quad
    &  (d,\delta,\beta,\gamma) \in P_\Delta,\enskip  
     d \in \mathbb{Z}^{N \times N},\enskip \text{and}\enskip  d_{i,j} = \tilde d_{i,j} \text{ for } i \in \{2,4, \ldots, N\}, j \in  \{1,\ldots,N\}.
\end{aligned}
\tag{red IP}
\end{gather}
The graph construction is similar to the one in \cite{severitt2023efficient} as we determine the control values of the nodes starting in the first row and continuing through the rows not already fixed but with two changes. The first change is that we need to consider the jumps to already fixed nodes given by $\tilde d_{i,j}$ which we do by adjusting the weights of the edges. The second change is that we need to adjust the weights as we start fixing a new row to model that we do not consider jumps from the last node of the previous row to the first node of the current row.

We construct a graph $G(V,A)$ with $ \frac{N^2}{2} \Xi \Delta +2$ nodes including the source and the sink. There are $\frac{N^2}{2}$ layers with respectively $\Xi \Delta$ nodes where the nodes of the first layer are connected to the source and the nodes of the last layer are connected to the sink. Each node is only connected to nodes in the previous and following layer. We describe a node $v \in V$, excluding the source and the sink, as in the previous subsection as a triplet $v = (j,\delta, \eta) \in [\frac{N^2}{2}] \times \{x-y | x,y \in \Xi^N \} \times \{0, \ldots, \Delta\}.$ We define the notation $\ell(v)= j$, $p(v) = \delta$ and $\tilde r(v) = \eta.$ 

An edge $e$ exist between two nodes $u,v \in V \setminus \{s,t\}$ is defined by
\[ e=(u, v) \in A\quad:\Longleftrightarrow\quad
\left\{
\begin{matrix}
    \ell(v) = \ell(u) + 1,&\\
    \text{there exists } (a,b) \in A \text{ with } b = u,
    &\\
    \tilde{r}(v) = \tilde{r}(u)
                     - \gamma_{\ell(v)}\vert p(v)\vert.& \\
\end{matrix}
\right.
\]
The first condition enforces the layer structure, while the second and last conditions ensure that the capacity constraint holds inductively. For a clearer presentation, we introduce $i(v) = a $ and $j(v) = b$ where $\ell(v) = (2a-1) N+b$. The weight of an edge $e=(u,v) \in A$ with  $u,v \in V \setminus \{s,t\}$ is given by 

\begin{align*}
    w_{(u,v)} =  &c_{i(v),j(v)} + \begin{cases} |x_{i(v),j(v)} + p(v) - x_{i(u),j(u)} -p(u) | \quad &\text{if } i(u)=i(v), \\
0 \quad &\text{otherwise,}
\end{cases}\\ &+ \begin{cases} |x_{i(v),j(v)} + p(v) - x_{i(v)+N ,j(v)} -\tilde d_{i(v)+N ,j(v)} | \quad &\text{if } i(v)<N,\\
0 \quad &\text{otherwise,}
\end{cases} \\  &+ \begin{cases}|x_{i(v),j(v)} + p(v) - x_{i(v)-N ,j(v)} -\tilde d_{i(v)-N ,j(v)} | \quad &\text{if } i(v)>1,\\
0 \quad &\text{otherwise.}
\end{cases}
\end{align*}
We note that the second case distinction is not needed because $N$ was assumed to be even thus the first case is always fulfilled but the distinction is done anyway for sake of completeness for the case of an odd $N$.
The 
source $s = (0,\emptyset,\Delta)$
is connected to all $v \in V$ in the first layer
with sufficient remaining capacity, that is
\[ (s,v) \in A\quad:\Longleftrightarrow\quad
   \ell(v) = 1 \text{ and }\tilde{r}(v) = \Delta - |p(v)| \gamma_1.
\]
The weight is given by $w_{(s,v)}=c_1 p_N(v)+ |x_{2,1} + \tilde d_{2,1} - x_{1,1} - p(v) | $. 
The sink $t = (N M+ 1,\emptyset,0)$ is connected to each node
$v \in V$ in the last layer that has an incoming edge, that is
\[ (v,t) \in A\quad:\Longleftrightarrow\quad
   \text{there exists } u \in V \text{ such that } (u,v) \in A. \]
The weights have the value zero, that is $w_{(v,t)} = 0$. Just like in \cite{severitt2023efficient} we can obtain an optimal solution for \eqref{eq:ip-1D} by solving the shortest path problem from $s$ to $t$. 

\begin{figure}
    \centering
\resizebox{0.275 \textwidth}{!}{%
\begin{tikzpicture}[node distance={65mm}, thick, main/.style = {draw, circle}] 
 \foreach \x in {0,...,3} {
 \foreach \y in {0,2} {
    \node[main, minimum size = 0.5 cm] at(\x ,\y ) {};
    
}
}
 \foreach \x in {0,...,3} {
 \foreach \y in {1,3} {
    \node[main, minimum size = 0.5 cm] at(\x ,\y ) {\scalebox{.7}{\tiny F}};
    
}
}

\foreach \x in {0,...,3} {
\foreach \y in {1} {
\draw (\x, \y+0.25) -- (\x,\y+0.75) node [below left] {\tiny C};
}
}
\foreach \x in {0,...,3} {
\foreach \y in {0,2} {
\draw (\x, \y+0.25) -- (\x,\y+0.75) node [below left] {\tiny B};
}
}
\foreach \x in {0,...,2} {
\foreach \y in {0,2} {
\draw (\x+0.25, \y) -- (\x+0.75,\y)node [below left] {\tiny A};
}
}
\foreach \x in {0,...,2} {
\foreach \y in {1,3} {
\draw (\x+0.25, \y) -- (\x+0.75,\y);
}
}
	\end{tikzpicture}
}

\caption{The nodes marked as $F$ are the fixed nodes in the grid. Each edge ending in at least one node not fixed beforehand represents one of the case distinction in the weights definition for the shortest path approach,
specifically A stands for the first, B for the second and C for the last case. }
\label{fig:grid primal improve}
\end{figure}
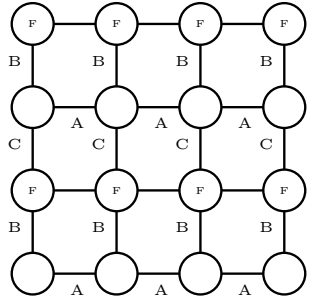

\bibliographystyle{abbrv} 
\bibliography{references}{}

\end{document}